\documentclass[11pt]{amsart}
\usepackage[margin=10pt,font=small,labelfont=normalfont]{caption}
\usepackage{subcaption}

\usepackage{mathtools}

\usepackage{hyperref}

\usepackage[demo]{graphicx}

\usepackage{xcolor}

\usepackage[margin=1.1in]{geometry} 

\usepackage{graphics,graphicx}
\usepackage{pstricks,pst-node,pst-tree}

\usepackage{xcolor}

\theoremstyle{plain}

\usepackage{amsfonts}
\usepackage{amssymb}
\usepackage{amsthm}
\usepackage{amsmath}
\usepackage{enumerate}
\usepackage{hyperref}
\hypersetup{
    colorlinks=false,
    pdfborder={0 0 0},
}
\usepackage{amsrefs}
\usepackage{graphicx}
\usepackage{verbatim}
\usepackage{amscd}
\usepackage{mathtools}

\usepackage{arydshln}

\theoremstyle{plain}

\newcommand{\IN}{\mathbb{N}}
\newcommand{\IZ}{\mathbb{Z}}

\def\sqr#1#2{{\,\vcenter{\vbox{\hrule height.#2pt\hbox{\vrule width.#2pt
height#1pt \kern#1pt\vrule width.#2pt}\hrule height.#2pt}}\,}}

\newtheorem{proposition}{Proposition}[section]
\newtheorem{diagram}{Diagram}
\newtheorem{lemma}[proposition]{Lemma}
\newtheorem{theorem}[proposition]{Theorem}
\newtheorem{corollary}[proposition]{Corollary}

\theoremstyle{definition}
\newtheorem{definition}[proposition]{Definition}
\theoremstyle{remark}

\newtheorem*{convention*}{Convention}
\newtheorem{example}[proposition]{Example}

\newtheorem{remark}[proposition]{Remark}

\numberwithin{equation}{section}

\numberwithin{equation}{section}

\title[Ternary rings of operators arising from inverse semigroups]{Ternary rings of operators arising from inverse semigroups}

\author{Robert Pluta}
\address{Department of Mathematics, University of California, Irvine, CA 92697-3875, USA,
and Wenzhou-Kean University}
\email{plutar@tcd.ie}

\author{Bernard Russo}
\address{Department of Mathematics, University of California, Irvine, CA 92697-3875, USA}
\email{brusso@math.uci.edu}

\date{\today}

\begin{document}
\begin{abstract}
We are interested in properties, especially {\it injectivity} (in the sense of category theory), of the ternary rings of operators generated by  certain subsets  of an inverse semigroup  via the regular representation. We determine all subsets of the extended bicyclic semigroup   which are closed under the triple product $xy^*z$ (called semiheaps) and show that the weakly closed ternary rings of operators generated by them  (W*-TROs) are injective operator spaces. \end{abstract}
\thanks{}
\maketitle
\tableofcontents

\section{Introduction and preliminaries}

\subsection{Introduction}
Ternary rings of operators originated in the work of M.\ R.\ Hestenes in 1962 \cite{Hestenes62}. These are linear spaces of operators from one Hilbert space to another which are stable under the triple product $XY^*Z$,  which he called {\it ternary algebras}. By their nature, these spaces satisfied an associativity condition involving five elements, namely,
\begin{equation}\label{eq:0121221}
(XY^*Z)U^*W=XY^*(ZU^*W)=X(UZ^*Y)^*W.
\end{equation}
These were subsequently axiomatized and  named {\it associative triple systems} \cite{Loos72}.  A milestone in their development in the realm of functional analysis was a Gelfand-Naimark type representation theorem for associative triple systems equipped with an operator type norm \cite{Zettl83}.

At around the same time as Hestenes' work,  unbeknownst to the researchers in the west due partially to the Cold War \cite{Hollings2014},  the concept of  {\it semiheap} was introduced in the Soviet Union \cite{Lawson2017}.  A semiheap is a set together with a single three variable operation satisfying an abstract version of (\ref{eq:0121221}), and akin to the known concepts of  {\it ternary group} and {\it inverse semigroup}.
Since the concept of semiheap is central to this paper, we provide the formal definition, as stated in \cite[p.\ 56]{Lawson2017}.
\begin{definition}
By a {\it semiheap}, we mean a set $K$ together with a singled-valued, everywhere defined
ternary operation $[\cdot\cdot\cdot]$, satisfying the condition
\[
[[k_1k_2k_3]k_4k_5]=[k_1[k_4k_3k_2]k_5]=[k_1k_2[k_3k_4k_5]].
\]
 \end{definition}

Semiheaps and their associated structures are closely related to inverse semigroups. In turn, inverse semigroups, together with groupoids, give rise to operator algebras \cite{Paterson1999}.  A ubiquitous example of an inverse semigroup is the {\it bicyclic semigroup} \cite[p.\ 188]{Paterson1999}. 

In this paper, we generalize the bicyclic semigroup in such a way that exhibits its semiheap structure.\footnote{After completing the final draft of this paper, the authors found out that this extended bicyclic semigroup has already been defined, see for example \cite{Warne1968}. However,  only binary structures are considered in this and other papers.} In Theorem~\ref{thm:0211221}, we  classify all of the subsemiheaps of this extended bicyclic semigroup. We then show in Corollary~\ref{cor:0211221}, via a general result applying to all  inverse semigroups \cite[Theorem 4.5.2]{Paterson1999},  that each of the examples resulting from this classification has the property that the weakly closed ternary ring of operators it generates is an injective operator space.  It is worth pointing out that, although the injectivity of the W*-TROs generated  by the classification of subsemiheaps uses  deep results in functional analysis (\cite[Theorem 2.5]{EffOzaRua01}, \cite[Theorem 4.5.2]{Paterson1999}), the classification itself is self-contained using only elementary arguments.

All of the subsemigroups of the bicyclic semigroup have been determined in \cite{Descalco2005}, while the sub-inverse semigroups were determined earlier in \cite{Schein2003}
and recently in \cite{Hovsepyan2020}. Since inverse subsemigroups are semiheaps, our results give a new approach to  the result in the latter two papers.

\subsection{Inverse semigroups}
 
An {\em inverse semigroup} is a semigroup $S$ in which for every element $x$ there exists a unique element $x^*$, called the {\em inverse} or {\em generalized inverse} of $x$, such that $x = x x^* x$ and $x^* = x^* x x^*$. The function that sends $x$ to $x^*$ is an antiautomorphism of order two. An {\em idempotent} in a semigroup is an element that multiplied by itself results in the same element.  An inverse semigroup may contain many idempotents,  but when it contains precisely one idempotent, it is a group. In fact,  the existence of a unique idempotent characterizes groups among inverse semigroups. 

\begin{remark}
Inverse semigroups are precisely the semigroups   
in which every element has at least one generalized inverse 
and in which any two idempotents commute. 
\end{remark}

\begin{example}    
The {\em symmetric inverse semigroup} $I(X)$,     
also denoted by $I_X$, which is the   
monoid under composition of all partial bijections (or partial symmetries) on a set $X$
is an inverse semigroup.  
The idempotents of $I(X)$  correspond to subsets $A \subseteq X$ 
via $A \leftrightarrow 1_A$, where $1_A$ is the identity function on $A$. 
In particular, $1_\emptyset$ is the empty function (zero element) 
and $1_X$ is the identity element of $I(X)$.   
The symmetric group on $X$ is a subgroup of $I(X)$. 
\end{example} 
 A classical result originating in the work of  Wagner and Preston, known as the Wagner-Preston Representation Theorem, states that every inverse semigroup $S$ may be embedded in the symmetric inverse semigroup $I(S)$. In fact, for any $a$ in $S$, the mapping   
$$ a^* S \to aS, \qquad x\mapsto ax \qquad \text{ ($x  \in a^* S$)}  $$
is a bijection from $a^* S$ onto $aS$ that represents the element $a$.  

\subsection{The extended bicyclic semigroup}\ 
We use the following notation: 
$\IN=\{1,2,\ldots\}$;  $\IN_0=\IN\cup\{0\}$;
$\IZ=\IN_0\cup -\IN$.



\begin{example}\label{exam:0916212}
Let $E_{22}$ be the bicyclic  semigroup, as realized by the unilateral shift (\cite[p. 188]{Paterson1999}). Thus $E_{22}=\{a_{ij}:i,j\in \IN_0\}$, where  for any $i,j\in\IZ$, $e_{ij}$ is the  matrix over $\IZ$  with 1 in the 
$i,j$ position and zeros elsewhere, and $a_{ij}$ is the bounded operator on the Hilbert space  $\ell^2(\IZ)$, 
\[
a_{ij}=\sum_{k\ge 0}e_{i+k,j+k}\in  B(\ell^2(\IZ)).
\]

Set $$E=E_{11}\cup E_{12}\cup E_{21}\cup E_{22}$$
where  $E_{21}=\{a_{ij}:i\in \IN_0, j\in -\IN\}$, $E_{11}=\{a_{ij}:i,j\in -\IN\}$ and  $E_{12}=\{a_{ij}:i\in -\IN, j\in \IN_0\}$.  

\end{example}

We note that for $i,j,p,q\in\IZ$, $a_{ij}^*=a_{ji}$, $a_{ij}$ is a partial isometry in $B(H)$ ($H=\ell^2(\IZ)$) and
\begin{equation}\label{eq:0927211}
a_{ij}a_{pq}=\left\{ \begin{array}{cc} a_{i,q+j-p}, & p\le  j\\
a_{i+p-j,q},& p\ge j
\end{array},
\right.
\end{equation}
equivalently $$a_{ij}a_{pq}=a_{i+p-\min(j,p),j+q-\min(j,p)}.$$

In particular, $a_{ij}a_{pq}\ne 0$, $a_{ij}a_{jq}=a_{iq}$, and $a_{ii}a_{pp}=a_{mm}$ with $m=\max(i,p)$.

Thus
$E$ is an inverse semigroup consisting of partial isometries with  ``inverse'' $a_{ij}^*$ equal to the adjoint of $a_{ij}$ and the idempotents coincide with the elements $a_{ii}$.  We shall call $E$ the {\it extended bicyclic  semigroup}\footnote{Ibid.}, and  when convenient notationally, represent $a_{ij}$  in formulas and diagrams simply by $(i,j)\in \IZ\times \IZ$.

We shall analyze the extended bicyclic semigroup $E$ of  Example~\ref{exam:0916212} toward the aims of finding all of the subsemiheaps of $E$, and showing that the associated W*-TROs are injective operator spaces.

  The idempotents of $E$ are the elements $a_{ii}$ with $i\in\IZ$ and $a_{ii}\le a_{jj}$, that is, $a_{ii}a_{jj}=a_{ii}$,  if and only if $j\le i$. From (\ref{eq:0927211}), we calculate and find that for $p,q\in\IZ$,
\[
a_{ii}a_{pq}=\left\{
 \begin{array}{ll}
  a_{i,q+i-p}\quad ,\quad&p\le i\\
  a_{pq}\quad ,\quad&p\ge i 
  \end{array}\right.,
\]
\[
a_{pq}a_{jj}=\left\{
 \begin{array}{ll}
  a_{pq}\quad ,\quad&j\le q\\
  a_{p+j-q,j}\quad ,\quad&j\ge q 
  \end{array}\right.,
\]
and  
  \[
  a_{ii}a_{pq}a_{jj}=
  \left\{
  \begin{array}{ll}
  a_{i,q+i-p}\quad ,\quad&p\le i\hbox{ and }j\le  q+i-p\\
  a_{j-q+p,q}\quad ,\quad&p\le i\hbox{ and }j\ge q+i-p\\
  a_{pq}\quad,\quad&p\ge i\hbox{ and }j\le q\\
  a_{p+j-q,j}\quad ,\quad&p\ge i\hbox{ and }j\ge q
  \end{array}\right.
  \]
In particular, $a_{00}E=E_{21}\cup E_{22}$ and $Ea_{00}
=E_{12}\cup E_{22}$. Also, $a_{ii}E$, $Ea_{jj}$ and $a_{ii}Ea_{jj}$
are subsemigroups and semiheaps, and  $a_{ii}Ea_{jj}$ is an inverse semigroup if $i=j$. More generally,
\[
Ea_{jj}=\{a_{pq}:j\le q\},\quad a_{ii}E=\{a_{pq}:p\ge i\},
\]
and
\[
a_{ii}E\cap Ea_{jj}=a_{ii}Ea_{jj}=\{a_{pq}:p\ge i, q\ge j\}.
\]

\section{Diagrams and examples}

In order to analyze the subsemiheaps of the extended bicyclic  semigroup $E$ in Example~\ref{exam:0916212}, we prepare some material. From (\ref{eq:0927211}), we have the following lemma.

\begin{lemma}\label{lem:0925211}
For any $a_{ij}, a_{pq}, a_{rs}$ in $E$, we have
\[
a_{ij}a_{pq}^*a_{rs}=\left\{
\begin{array}{lrl}
\hbox{(i) }a_{i,s+p+j-q-r}&\quad r\le p+j-q,&q\le j\\
\hbox{(ii) }a_{i+r-p-j+q,s}&\quad r\ge p+j-q,&q\le j\\
\hbox{(iii) } a_{i+q-j+r-p,s}&\quad r\ge p,&q\ge j\\
\hbox{(iv) }a_{i+q-j,s+p-r}&\quad r\le p,&q\ge j\\
 \end{array}\right.
\]
\end{lemma}

It is worth noting, as will be evident in the ten diagrams that follow, all triple products in $E$ which involve only two elements, produce new elements which propagate  only  to the right  of,  and  down from,  the two elements, and not to the left or up.

Lemma~\ref{lem:0925212}  and Diagrams 1-5 describe the case in which the slope of the line connecting the two points is negative (or zero or infinite). Lemma~\ref{lem:0929211} and Diagrams 6-10 describe the case in which the slope of the line connecting the two points is positive (or zero or infinite).

\begin{lemma}\label{lem:0925212}
If $K$ is a subsemiheap of $E$, and if $a_{\alpha\beta},a_{\gamma\delta}\in K$ with $\gamma\ge \alpha$ and $ \delta\ge \beta$, then the following elements belong to $K$:
\begin{itemize}
\item $x_1=a_{\alpha+\delta-\beta,\beta+\gamma-\alpha}$
\item $x_2=a_{\alpha+\delta-\beta,\delta}$
\item $x_3=a_{\gamma,\beta+\gamma-\alpha}$
\item $x_4=a_{\gamma,\delta+(\delta-\beta)-(\gamma-\alpha)}$\quad if $\gamma-\alpha\le\delta-\beta$
\item $x_5=a_{\gamma+(\gamma-\alpha)-(\delta-\beta),\delta}$ \quad if $\gamma-\alpha\ge\delta-\beta$
\end{itemize}
\end{lemma}
\begin{proof} The following are the eight possible triple products containing two distinct elements, and thus belong to $K$. They are calculated using Lemma~\ref{lem:0925211}.
\begin{itemize}
\item $a_{\alpha\beta}a_{\alpha\beta}^*a_{\alpha\beta}=a_{\alpha\beta}$
\item $a_{\alpha\beta}a_{\gamma\delta}^*a_{\alpha\beta}=a_{\alpha+\delta-\beta,\beta+\gamma-\alpha}=x_1$
\item $a_{\alpha\beta}a_{\alpha\beta}^*a_{\gamma\delta}=a_{\gamma\delta}$
\item $a_{\alpha\beta}a_{\gamma\delta}^*a_{\gamma\delta}=a_{\alpha+\delta-\beta,\delta}=x_2$
\item $a_{\gamma\delta}a_{\alpha\beta}^*a_{\alpha\beta}=a_{\gamma\delta}$
\item $a_{\gamma\delta}a_{\gamma\delta}^*a_{\alpha\beta}=a_{\gamma,\beta+\gamma-\alpha}=x_3$
\item $a_{\gamma\delta}a_{\alpha\beta}^*a_{\gamma\delta}=
\left\{\begin{array}{ll}
a_{\gamma,\delta+\alpha+\delta-\beta-\gamma}=x_4,&\quad \gamma-\alpha\le\delta-\beta\\
a_{\gamma+\gamma-\alpha-\delta+\beta,\delta}=x_5, &\quad  \gamma-\alpha\ge\delta-\beta
\end{array}\right.$
\item $a_{\gamma\delta}a_{\gamma\delta}^*a_{\gamma\delta}=a_{\gamma\delta}$
\end{itemize}

\end{proof}

Special cases of Lemma~\ref{lem:0925212} and their diagrams, which serve as useful references, are as follows:

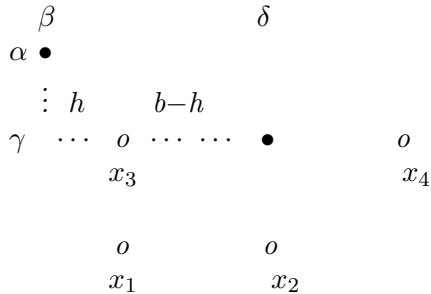
\begin{diagram}\label{diagram:1022211}
\quad\quad$ \delta-\beta>\gamma-\alpha>0$\quad ($b=\delta-\beta, h=\gamma-\alpha$)\bigskip

\hspace{2.15in}$\beta$\hspace{1in}   $\delta$

\hspace{2in}$\alpha$\ $\bullet$

\vspace{.05in}

\hspace{2.165in}$\vdots$\hspace{.1in}$h$\hspace{.35in}b$-$h\hspace{.2in}

\hspace{2in}$\gamma$\ \ \   $\cdots$\hspace{.1in}o \ $\cdots\ \cdots$\hspace{.12in}$\bullet$\hspace{.63in}o

\hspace{2.52in}$x_3$\hspace{1.33in}\ $x_4$
\vspace{.2in}

\hspace{2.4in}\hspace{.15in}o\hspace{.7in}o

\hspace{2.52in}$x_1$\hspace{.7in}$x_2$
\end{diagram}
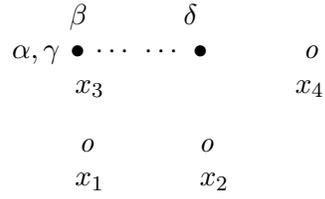
\bigskip

\begin{diagram}\label{diagram:1022212}
\quad\quad$\gamma= \alpha, \delta-\beta>0$\bigskip

\hspace{2.3in}$\beta$\hspace{.45in}  $\delta$

\hspace{2in}$\alpha,\gamma$ $\bullet$ $\cdots\ \cdots$ $\bullet$\hspace{.5in}o

\hspace{2.33in}$x_3$\hspace{1in}$x_4$
\vspace{.13in}

\hspace{2.35in}o\hspace{.55in}o

\hspace{2.33in}$x_1$\hspace{.5in}$x_2$
\end{diagram}
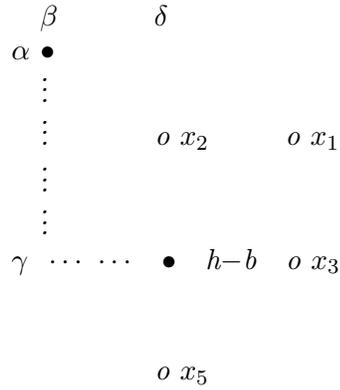
\bigskip

\begin{diagram}\label{diagram:1022213}
\quad\quad$\gamma-\alpha>\delta-\beta>0$\quad ($b=\delta-\beta, h=\gamma-\alpha$)
\bigskip

\hspace{2.in}$\beta$\hspace{.45in}  $\delta$

\hspace{1.85in}$\alpha$ $\bullet$

\hspace{2in}$\vdots$

\hspace{2in}$\vdots$
\hspace{.5in}o $x_2$\hspace{.4in}o $x_1$

\hspace{2in}$\vdots$

\hspace{2in}$\vdots$         \hspace{.05in}\hspace{.05in}

\hspace{1.85in}$\gamma$\ \ $\cdots\ \cdots$
\hspace{.08in}$\bullet$
 \hspace{.1in}h$-$b\hspace{.1in} o $x_3$
\vspace{.4in}

\hspace{2in}\hspace{.6in}o $x_5$

\end{diagram}

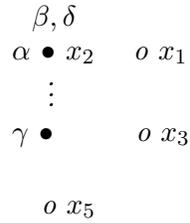
\begin{diagram}\label{diagram:1022214}
\quad\quad$\delta=\beta,\gamma-\alpha>0$\bigskip

\bigskip

\hspace{2.05in} $\beta,\delta$

\hspace{2in}$\alpha$ $\bullet$ $x_2$\hspace{.2in}o $x_1$

\hspace{2.18in}$\vdots$

\hspace{2in}$\gamma$ $\bullet$\hspace{.43in}o $x_3$
\vspace{.2in}

\hspace{2.1in} o   $x_5$
\end{diagram}

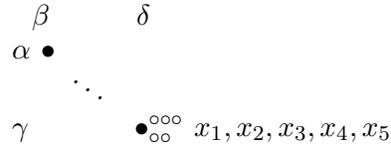
\bigskip

\begin{diagram}\label{diagram:1022215}
\quad\quad$\delta-\beta=\gamma-\alpha>0$\bigskip

\hspace{2.05in} $\beta$\hspace{.4in}    $\delta$

\hspace{2in}$\alpha$ $\bullet$\hspace{.3in}

\hspace{2.3in}$\ddots$

\hspace{2in}$\gamma$ \hspace{.5in}$\bullet_{\circ\circ}^{\circ\circ\circ}$ $x_1,x_2,x_3,x_4,x_5$
\end{diagram}

\bigskip

\begin{lemma}\label{lem:0929211}
If $K$ is a subsemiheap of $E$, and if $a_{\alpha\beta},a_{\gamma\delta}\in K$ with $\gamma\ge \alpha, \delta\le \beta$, then the following elements belong to $K$:
\begin{itemize}
\item $x_1=a_{\alpha,\beta+(\gamma-\alpha)+(\beta-\delta)}$
\item $x_2=a_{\gamma+\beta-\delta,\beta}$
\item $x_3=a_{\gamma,\beta+\gamma-\alpha}$
\item $x_4=a_{\gamma+(\beta-\delta)+(\gamma-\alpha),\delta}$
\end{itemize}
\end{lemma}
\noindent{\it Proof.}
The following eight products belong to $K$ and can be calculated using Lemma~\ref{lem:0925211}.
\begin{itemize}
\item $a_{\alpha\beta}a_{\alpha\beta}^*a_{\alpha\beta}=a_{\alpha\beta}$
\item $a_{\alpha\beta}a_{\gamma\delta}^*a_{\alpha\beta}=a_{\alpha,\beta+(\gamma-\alpha)+(\beta-\delta)}=x_1$
\item $a_{\alpha\beta}a_{\alpha\beta}^*a_{\gamma\delta}=a_{\gamma\delta}$
\item $a_{\alpha\beta}a_{\gamma\delta}^*a_{\gamma\delta}=a_{\alpha\beta}$
\item $a_{\gamma\delta}a_{\alpha\beta}^*a_{\alpha\beta}=a_{\gamma+\beta-\delta,\beta}=x_2$
\item $a_{\gamma\delta}a_{\gamma\delta}^*a_{\alpha\beta}=a_{\gamma,\beta+\gamma-\alpha}=x_3$
\item $a_{\gamma\delta}a_{\alpha\beta}^*a_{\gamma\delta}=a_{\gamma+\beta-\delta+\gamma-\alpha,\delta}=x_4$
\item $a_{\gamma\delta}a_{\gamma\delta}^*a_{\gamma\delta}=a_{\gamma\delta}$\hfill$\Box$
\end{itemize}

Special cases of Lemma~\ref{lem:0929211} and their diagrams are as follows:

\begin{diagram}\label{diagram:1027211}
\quad\quad$ \beta-\delta>\gamma-\alpha>0$\quad ($b=\beta-\delta, h=\gamma-\alpha$)\bigskip

\hspace{2.15in}$\delta$\hspace{.83in}   $\beta$

\hspace{2in}$\alpha$\hspace{1in}$\bullet$ \hspace{1.2in} o $x_1$

\hspace{3.05in} $\vdots$ \ \ $h$ \hspace{1in} $\vdots$

\hspace{2in}$\gamma$\ \  $\bullet$ $\cdots$\ $\cdots$\  $\cdots$  \hspace{.4in} $ \circ$ $\cdots$\ $\cdots$\ $\cdots$

\hspace{3.5in} $x_3$
\vspace{.21in}

\hspace{2in} $b+h$
\vspace{.25in}

\hspace{2in} \hspace{1.08in}o $x_2$

\hspace{3.17in}$\vdots$

\hspace{2.2in}o\  \ \   $\cdots\ \cdots\ \cdots$

\hspace{2.1in} $x_4$

\end{diagram}

\begin{diagram}\label{diagram:1027212}
\quad\quad$\gamma= \alpha, \beta-\delta>0$\bigskip

\hspace{2.3in}$\delta$\hspace{.49in}  $\beta$

\hspace{2in}$\alpha,\gamma$ $\bullet$ $\cdots\ \cdots$ $\bullet$\hspace{.5in}o

\hspace{2.95in}$x_3$\hspace{.44in}$x_1$
\vspace{.13in}

\hspace{2.35in}o\hspace{.55in}o

\hspace{2.33in}$x_4$\hspace{.5in}$x_2$
\end{diagram}

\begin{diagram}\label{diagram:1027213}
\quad\quad$\gamma-\alpha>\beta-\delta>0$\quad ($b=\beta-\delta, h=\gamma-\alpha$)
\bigskip

\hspace{2in}$\delta$\hspace{.23in}  $\beta$\hspace{.45in} $b+h$

\hspace{1.85in}$\alpha$\hspace{.35in} $\bullet$\hspace{1.2in} o $x_1$

\hspace{2.35in}$\vdots$\hspace{1.25in} $\vdots$

\hspace{2.35in}$\vdots$\hspace{1.25in} $\vdots$

\hspace{2.35in}$\vdots$\hspace{1.25in} $\vdots$

\hspace{1.85in}$\gamma$\ \  $\bullet$ $\cdots$\hspace{.08in}\hspace{.85in} o $\cdots$

\hspace{2in} \hspace{1.28in}$x_3$
\vspace{.05in}

\hspace{2.3in} o $x_2$

\hspace{2.4in}$\vdots$

\hspace{2.4in}$\vdots$

\hspace{2.4in}$\vdots$

\hspace{1.8in}$x_4$\hspace{.12in}o $\cdots$


\end{diagram}


\begin{diagram}\label{diagram:1027214}
\quad\quad$\delta=\beta,\gamma-\alpha>0$\bigskip

\bigskip

\hspace{2.05in} $\beta,\delta$

\hspace{2in}$\alpha$ $\bullet$ \hspace{.35in}o $x_1$

\hspace{2.18in}$\vdots$

\hspace{2in}$\gamma$ $\bullet$ $x_2$\hspace{.23in}o $x_3$
\vspace{.24in}

\hspace{2in}o $x_4$
\end{diagram}

\bigskip

\begin{diagram}\label{diagram:1027215}
\quad\quad$\delta-\beta=\gamma-\alpha>0$\bigskip

\hspace{2.05in} $\delta$\hspace{.4in}    $\beta$

\hspace{2in}$\alpha$\hspace{.47in} $\bullet$\hspace{.95in} o $x_1$

\hspace{2.6in} $\vdots$\hspace{.95in} $\vdots$

\hspace{2in}$\gamma$ $\bullet$ \ \ $\cdots$\ \ \hspace{.5in}o\ \  $\cdots$

\hspace{3.1in} $x_3$

\vspace{.23in}

\hspace{2.65in}o $x_2$

\hspace{2.6in} $\vdots$

\hspace{2.1in} o\ \ \ $\cdots$

\hspace{2.1in} $x_4$

\end{diagram}

\bigskip

\begin{example}
For $\alpha,\beta\in \IZ$, and $J\subset\IN_0$,  $D_{\alpha,\beta}(J):= \{a_{\alpha+j,\beta+j}: j\in J\}$ is a subsemiheap of $E$.
\end{example}

\begin{example}
For $\alpha,\beta\in \IZ$, and $\sigma\in\IN$, 
$K_{\alpha,\beta}= \{a_{\alpha+\ell,\beta+m}: \ell,m\in\IN_0\}$, and more generally,
 $K_{\alpha,\beta}^\sigma:= \{a_{\alpha+\ell\sigma,\beta+m\sigma}: \ell,m\in\IN_0\}$ are subsemiheaps of $E$.
\end{example}

\begin{example}\label{ex:1201211}
  $
  K=\{a_{\alpha_0\beta_0}\}\cup\{a_{\alpha_0+k_j,\beta_0+k_j}:j=1,\ldots \ell_0-1\}\cup K_{\alpha_0+k_{\ell_0},\beta_0+k_{\ell_0}}^\sigma 
  $  is a subsemiheap of $E$, where $\alpha_0,\beta_0\in\IZ$, $\ell_0,\sigma,k_i\in\IN$, $k_1<k_2<\cdots<k_{\ell_0-1}$, and 
  \[
  K_{\alpha_0+k_{\ell_0},\beta_0+k_{\ell_0}}^\sigma =\{a_{\alpha_0+k_{\ell_0}+m\sigma,\beta_0+k_{\ell_0}+n\sigma}:m,n\in\IN_0\}.
  \]
  
\end{example}
\begin{proof}
Let $x_j=a_{\alpha_0+k_j,\beta_0+k_j}$ for $1\le k_j<\ell_0$ and $y_{mn}=a_{\alpha_0+\ell_0+m\sigma,\beta_0+\ell_0+n\sigma}$ for $m,n\in\IN_0$. 

The following eight products belong to $K$, as calculated by Lemma~\ref{lem:0925211}.
\begin{enumerate}
\item $x_jy_{mn}^*y_{pq}=\left\{ \begin{array}{ll}
y_{n+p-m,q}&\hbox{ if } p\ge m \quad\quad \hbox{ (Lemma~\ref{lem:0925211}(iii))}\\
y_{n,q+m-p}&\hbox{ if }  p\le m  \quad\quad\hbox{ (Lemma~\ref{lem:0925211}(iv))}\\
\end{array}\right.$
\item $y_{mn}x_j^*y_{pq}=\left\{ \begin{array}{ll}
y_{m,n+q-p}&\hbox{ if } p\le n  \quad\quad\hbox{ (Lemma~\ref{lem:0925211}(i))}\\
y_{m+p-n,q}&\hbox{ if }  p\ge n   \quad\quad\hbox{ (Lemma~\ref{lem:0925211}(ii))}\\
\end{array}\right.$

\item $y_{mn}y_{pq}^*x_j=\left\{ \begin{array}{ll}
y_{m,p+n-q}&\hbox{ if } q\le n  \quad\quad \hbox{ (Lemma~\ref{lem:0925211}(i))}\\
y_{m+q-n,p}&\hbox{ if }  q\ge n  \quad\quad \hbox{ (Lemma~\ref{lem:0925211}(iv)}\\
\end{array}\right.$
\item $y_{mn}x_i^*x_j=y_{mn}    \quad\quad\hbox{ (Lemma~\ref{lem:0925211}(i))}$
\item $x_iy_{mn}^*x_j=y_{nm} \quad\quad\hbox{ (Lemma~\ref{lem:0925211}(iv))}$
\item $x_ix_j^*y_{mn}=y_{mn} \quad\quad\hbox{ (Lemma~\ref{lem:0925211}(ii))}$
\item $x_ix_j^*x_\ell=x_{\max(i,j,\ell)}\quad\hbox{ (Lemma~\ref{lem:0925211}(i)-(iv))}$
\item $y_{mn}y_{pq}^*y_{rs}=\left\{ \begin{array}{ll}
y_{m,s+p+n-r-q}&\hbox{ if } q\le n\hbox{ and }r+q\le p+n\quad\hbox{ (Lemma~\ref{lem:0925211}(i))}\\
y_{ m+r+q-p-n,s}&\hbox{ if }  q\le n\hbox{ and }r+q\ge p+n\quad\hbox{ (Lemma~\ref{lem:0925211}(ii))}\\
y_{m+q-n+r-p,s}&\hbox{ if } q\ge n\hbox{ and }r\ge p\quad\quad\quad\quad\hbox{ (Lemma~\ref{lem:0925211}(iii))}\\
y_{ m+q-n+p-r,s}&\hbox{ if }  q\ge n\hbox{ and }r\le p\quad\quad\quad\quad\hbox{ (Lemma~\ref{lem:0925211}(iv))}\\
\end{array}\right.$
\end{enumerate}

We provide some details for cases (1) and (8).
For case (1), by Lemma~\ref{lem:0925211}(iii),
\[
x_jy_{mn}^*y_{pq}=
a_{\alpha_0+k_j,\beta_0+k_j}  a_{\alpha_0+\ell_0+m\sigma,\beta_0+\ell_0+n\sigma}^*
a_{\alpha_0+\ell_0+p\sigma,\beta_0+\ell_0+q\sigma}=a_{\alpha_0+\ell_0+(n+p-m)\sigma,\beta_0+\ell_0+q\sigma},
\]
if $p\ge m$ and $\ell_0+n\sigma\ge k_j$, 
and by Lemma~\ref{lem:0925211}(iv),
\[
a_{\alpha_0+k_j,\beta_0+k_j}  a_{\alpha_0+\ell_0+m\sigma,\beta_0+\ell_0+n\sigma}^*
a_{\alpha_0+\ell_0+p\sigma,\beta_0+\ell_0+q\sigma}=a_{\alpha_0+\ell_0+n\sigma,\beta_0+\ell_0+(q+m-p)\sigma},
\]
if $p\le m$ and $\ell_0+n\sigma\ge k_j$.\smallskip

For case (8), with$q\le n$, by Lemma~\ref{lem:0925211}(i),
\[
a_{\alpha_0+\ell_0+m\sigma,\beta_0+\ell_0+n\sigma}  a_{\alpha_0+\ell_0 +p\sigma,\beta_0+\ell_0+q\sigma}^*
a_{\alpha_0+\ell_0+r\sigma,\beta_0+\ell_0+s\sigma}=a_{\alpha_0+\ell_0+m\sigma,\beta_0+\ell_0+(s+p+n-r-q)\sigma},
\]
if $q\le n$ and $r+q\le p+n$, so that $s+p+n-r-q\in\IN_0$; and by Lemma~\ref{lem:0925211}(ii),
\[
a_{\alpha_0+\ell_0+m\sigma,\beta_0+\ell_0+n\sigma}  a_{\alpha_0+\ell_0 +p\sigma,\beta_0+\ell_0+q\sigma}^*
a_{\alpha_0+\ell_0+r\sigma,\beta_0+\ell_0+s\sigma}=a_{\alpha_0+\ell_0+(m+r+q-p-n)\sigma,\beta_0+\ell_0+s\sigma},
\]
if $q\le n$ and $r+q\ge p+n$, so that $m+r+q-p-n\in\IN_0$.\smallskip

The subcases  of (8) for which $q\ge n$ are as follows. 
By  Lemma~\ref{lem:0925211}(iii),
\[
a_{\alpha_0+\ell_0+m\sigma,\beta_0+\ell_0+n\sigma}  a_{\alpha_0+\ell_0 +p\sigma,\beta_0+\ell_0+q\sigma}^*
a_{\alpha_0+\ell_0+r\sigma,\beta_0+\ell_0+s\sigma}=a_{\alpha_0+\ell_0+(m+q-n+r-p)\sigma,\beta_0+\ell_0+s\sigma},
\]
if $q\ge n$ and $r\le p$, so that $m+q-n\in\IN_0$; and by Lemma~\ref{lem:0925211}(iv),
\[
a_{\alpha_0+\ell_0+m\sigma,\beta_0+\ell_0+n\sigma}  a_{\alpha_0+\ell_0 +p\sigma,\beta_0+\ell_0+q\sigma}^*
a_{\alpha_0+\ell_0+r\sigma,\beta_0+\ell_0+s\sigma}=a_{\alpha_0+\ell_0+(m+q-n)\sigma,\beta_0+\ell_0+(s+p-r)\sigma},
\]
if $q\ge n$ and $r\le p$, so that $m+q-n\in\IN_0$ and $s+p-r\in\IN_0$.
\end{proof}

\begin{example}\label{example:1220211}
Let $K= \bigcup_{k\in A} K_{\alpha_0+k,\beta_0+k}^p$, where $\alpha_0,\beta_0\in\IZ$, $p>0$, $A\subset \{0,1,\ldots, p-1\}$
and  $a_{\alpha_0+k,\beta_0+k}$, $k\in A$, denote  the elements of $K$ lying on the diagonal with $k<p$.  
(See Diagram~\ref{diagram:0308221} and Proposition~\ref{prop:1213211}.) In fact, $K$ is a sub-inverse semigroup of $E$.
\end{example}
\begin{proof} We note first that (setting $\alpha_0=\beta_0=0$ for convenience, see Remark~\ref{rem:0315221})
\[
K=\{(k+\ell p,k+mp):k\in A,\ell,m\in\IN_0\}
\]
and it suffices to show that 
\[
(k_1+\ell_1p,k_1+m_1p)(k_2+\ell_2 p,k_2+m_2p)^*(k_3+\ell_3p,k_3+m_3p)
\]
belongs to $K$.  We calculate this triple product using the four cases in Lemma~\ref{lem:0925211}.\smallskip

By Lemma~\ref{lem:0925211}(i), if $k_2+m_2p\le k
_1+m_1p$, and $k_3+\ell_3p\le k_1+(\ell_2+m_1-m_2)p$, then 
\[
(k_1+\ell_1p,k_1+m_1p)(k_2+\ell_2 p,k_2+m_2p)^*(k_3+\ell_3p,k_3+m_3p)=(k_1+\ell_1p,k_1+(m_3+\ell_2+m_1-m_2-\ell_3)p),
\]
and it is required to show that $m_3+\ell_2+m_1-m_2-\ell_3\ge 0$.

Following the argument in \cite[Lemma 4.5]{Descalco2005},  we have
\[
k_1+(\ell_2+m_1-m_2-\ell_3)p\ge k_3\ge 0
\]
so that 
$(\ell_2+m_1-m_2-\ell_3)p\ge -k_1>-p$ and therefore $\ell_2+m_1-m_2-\ell_3\ge 0$ and $m_3+ \ell_2+m_1-m_2-\ell_3\ge 0$, as required.\smallskip

By Lemma~\ref{lem:0925211}(ii), if $k_2+m_2p\le k
_1+m_1p$, and $k_3+\ell_3p\ge k_1+(\ell_2+m_1-m_2)p$, then 
\[
(k_1+\ell_1p,k_1+m_1p)(k_2+\ell_2 p,k_2+m_2p)^*(k_3+\ell_3p,k_3+m_3p)=(k_3+(\ell_1+\ell_3-\ell_2-m_1+m_2)p,k_3+m_3p)
\]
and it is required to show that $\ell_3-\ell_2-m_1+m_2\ge 0$.

Following the argument in \cite[Lemma 4.5]{Descalco2005},  we have
\[
k_3+(\ell_3-\ell_2-m_1+m_2)p\ge k_1\ge 0
\]
so that 
$(\ell_3-\ell_2-m_1+m_2)p\ge -k_3>-p$ and therefore $\ell_3-\ell_2-m_1+m_2\ge 0$ and $\ell_1+\ell_3-\ell_2-m_1+m_2\ge 0$, as required.\smallskip

By Lemma~\ref{lem:0925211}(iii), if $k_3+\ell_3p\ge k
_2+\ell_2p$, and $k_2+m_2p\ge k_1+m_1p$, then 
\[
(k_1+\ell_1p,k_1+m_1p)(k_2+\ell_2 p,k_2+m_2p)^*(k_3+\ell_3p,k_3+m_3p)=(k_3+(\ell_1+m_2-m_1+\ell_3-\ell_2)p,k_3+m_3p)
\]
and it is required to show that $\ell_1+m_2-m_1+\ell_3-\ell_2\ge 0$.

Following the argument in \cite[Lemma 4.5]{Descalco2005},  we have
\[
k_3+(\ell_3-m_1+m_2-\ell_2)p\ge k_1\ge 0
\]
so that 
$(\ell_3-m_1+m_2-\ell_2)p\ge -k_3>-p$ and therefore $\ell_3-m_1+m_2-\ell_2\ge 0$ and $\ell_1+\ell_3-m_1+m_2-\ell_2\ge 0$, as required.\smallskip

By Lemma~\ref{lem:0925211}(iv), if $k_3+\ell_3p\le k
_2+\ell_2p$, and $k_2+m_2p\ge k_1+m_1p$, then 
\[
(k_1+\ell_1p,k_1+m_1p)(k_2+\ell_2 p,k_2+m_2p)^*(k_3+\ell_3p,k_3+m_3p)=(k_2+(\ell_1+m_2-m_1)p,k_2+(m_3+\ell_2-\ell_3)p)
\]
and it is required to show that $\ell_1+m_2-m_1\ge 0$. and $m_3+\ell_2-\ell_3\ge 0$.

Following the argument in \cite[Lemma 4.5]{Descalco2005},  we have
\[
k_2+(\ell_2-\ell_3)p\ge k_3\ge 0
\]
so that 
$(\ell_2-\ell_3)p\ge -k_3>-p$ and therefore $\ell_2-\ell_3\ge 0$ and $m_3 +\ell_2-\ell_3\ge 0$, as required.

Following the argument in \cite[Lemma 4.5]{Descalco2005},  we have
\[
k_2+(m_2-m_3)p\ge k_1\ge 0
\]
so that 
$(m_2-m_1)p\ge -k_1>-p$ and therefore $m_2-m_1\ge 0$ and $\ell_1+m_2-m_1\ge 0$, as required.
\end{proof}

\begin{remark}\label{rem:0315221}

The adjoint operation $a_{ij}\mapsto a_{ij}^*=a_{ji}$ on the extended bicyclic semigroup $E$ is an anti-isomorphism of a subsemiheap $K$  of $E$ onto the subsemiheap $K^*$, that is, $(ab^*c)^*=c^*ba^*$.
As another application of Lemma~\ref{lem:0925211}, 
the  translation map on the extended bicyclic semigroup is a triple isomorphism, that is, if $\varphi_{\alpha,\beta}(a_{ij})=a_{i+\alpha,j+\beta}$, then  
\[
\varphi(a_{ij}a_{pq}^*a_{rs})=\varphi( a_{ij})\varphi (a_{pq})^*\varphi (a_{rs}).
\]

Hence, if $K$ is a subsemiheap of $K_{\alpha,\beta}$, then $\varphi_{-\alpha,-\beta}(K)$ is a subsemiheap of the bicyclic semigroup $K_{0,0}$.   At the very least, this fact can simplify notation in parts of this paper.  A sub-semigroup of the bicyclic semigroup $K_{0,0}$ (see \cite{Descalco2005}), is not necessarily  a subsemiheap, unless it is a sub-inverse semigroup (cf. \cite[Corollary 7.1]{Descalco2005}). 

\end{remark}

\begin{diagram}\label{diagram:0308221}

{\footnotesize

\begin{center}
\begin{tabular}{| r | r |  c c  c c c c c c c c c c c c c c c c c c c c c c  c| c}\hline

& & & & & $k_1$ &  & $k_2$& $k_3$   & &$k_4$ &  &  $k_5$ & & $k_6$ & $k_7$ & & $k_8$ & & $k_9$ & & & $k_{10}$ & & $k_{11}$&  $k_{12}$ & &\\\hline

& & 0 &  & & $q$ & &  & $p$ & & & & $q+p$ & & & $2p$ & & & & $q+2p$ & & & $3p$ & & &  $q+3p$  &\\\hline

& 0 & $\bullet$ &  & & & & & $\blacktriangle$ & & & & & & & $\blacktriangle$  & & & & & & & $\blacktriangle$ & &   & &\\

& & &  & & & &  & & & & & & & & & & & & & & & & & &  &\\

& & &  & & & &  & & & & & & & & & & & & & & & & & &  &\\

$k_1$ & q & &  & &$\blacksquare$ & &  & & & & &$\blacksquare$   & & & & & & & $\blacksquare$ & & & & & &$\blacksquare$& $\cdots$ \\

& & &  & & & &  & & & & & & & & & & & & & & & & &  & &\\

$k_2$& & &  & & & & $\bullet$ & & & & & & & $\bullet$ & & & & & & & $\bullet$ & & & &  &\\

$k_3$ & $p$ & $\blacktriangle$  &  & & & &  & $\blacktriangle$ & & & & & & & $\blacktriangle$ & & & & & & & $\blacktriangle$& & &  &\\

& & & & &  & &  & & & & & & & & & & & & & & & & & &  &\\

$k_4$ & & & & &  & &  & & & $\bullet$ & & & & & & & $\bullet$ & & & & & & & $\bullet$ &  &\\

$k_5$ & $q+p$& & &  &  $\blacksquare$ &  &  & & & & & $\blacksquare$ & & & & & & & $\blacksquare$ & & & & & & $\blacksquare$  & $\cdots$\\

& & & & &  & &  & & & & & & & & & & & & & & & & & &  &\\

$k_6$ & & & &   & &  & & & & & & & & $\bullet$ & & & & & & & $\bullet$& & & &   &\\

$k_7$ & $2p$ & $\blacktriangle$  &  & & & &  & $\blacktriangle$ & & & & & & & $\blacktriangle$ & & & & & & & $\blacktriangle$& & &  &\\

& & & & &  & &  & & & & & & & & & & & & & & & & & &  &\\

$k_8$ & & & && &&   & &  & & & & & & & & $\bullet$ & & & & & & & $\bullet$   &  &\\

$k_9$& $q+2p$ & & &   & $\blacksquare$ &  & & & & & & $\blacksquare$  & & & & & & & $\blacksquare$ & & &  & & & $\blacksquare$ &$\cdots$\\

& & & & &  & &  & & & & & & & & & & & & & & & & & &  &\\

$k_{10}$ & $3p$ & $\blacktriangle$  &  & & & &  & $\blacktriangle$ & & & & & & & $\blacktriangle$ & & & & & & & $\blacktriangle$& & &  &\\

& & & & &  & &  & & & & & & & & & & & & & & & & & &  &\\

$k_{11}$ & & & & &  & &  & & & & & & & & & & & & & & & & &$\bullet$ &  &\\

& & & & &  & &  & & & & & & & & & & & & & & & & & &  &\\

$k_{12}$& $q+3p$ & & &   & $\blacksquare$ &  & & & & & & $\blacksquare$  & & & & & & & $\blacksquare$ & & &  & & & $\blacksquare$ &$\cdots$\\

& & & & &  & &  & & & & & & & & & & & & & & & & & &  &\\

& & & &   & $\vdots$  & & &  & & & &  $\vdots$ & & & & & & &  $\vdots$ & & & & & &   
$\vdots$  &\\\hline
\end{tabular}
\end{center}

}

\end{diagram}

\section{Subsemiheaps of the extended bicyclic  semigroup}

In this section we shall determine all of the subsemiheaps of the extended bicyclic semigroup.  We shall proceed as follows.  First, for an arbitrary subsemiheap $K$ of $E$, we define
\[
\alpha_0=\inf\{\alpha\in\IZ:\exists \beta\in\IZ, a_{\alpha\beta}\in K\},
\]
and
\[
\beta_0=\inf\{\beta\in\IZ:\exists \alpha\in\IZ, a_{\alpha\beta}\in K\}.
\]

We have four mutually exclusive and exhaustive cases, namely, 
\begin{enumerate}
\item {\bf Quadrant} $\alpha_0\ne-\infty, \beta_0\ne -\infty$
\item {\bf Right Half Plane} $\alpha_0=-\infty, \beta_0\ne -\infty$
\item {\bf Lower Half Plane} $\alpha_0\ne-\infty, \beta_0= -\infty$
\item {\bf Full Plane} $\alpha_0=-\infty, \beta_0= -\infty$
\end{enumerate}

\begin{remark}\label{rem:0308221}
We only need to find all of the subsemiheaps of $E$ which are in case (1), since the other cases can be reduced to this case in steps, as follows, which shows that every subsemiheap of the extended bicyclic semigroup is the inductive limit of subsemiheaps in case (1) in the category of semiheaps and semiheap homomorphisms\footnote{In fact, it is an elementary inductive limit since the connecting maps are inclusions (see Theorem~\ref{thm:0211221}).}
\begin{itemize}
\item  If a subsemiheap $K$ of $E$ is in case (2), then 
$K\subset\{a_{ij}:i\in\IZ, j\ge\beta_0\}$ and 
$K=\cup_{\alpha\in\IZ} K^\alpha$, where $K^{\alpha}=K\cap \{a_{ij}:i\ge\alpha, j\ge \beta_0\}$ (which we have denoted by $K_{\alpha,\beta_0}$) is in case (1).

\item If a subsemiheap $K$ of $E$ is in case (3), then 
$K\subset\{a_{ij}:i\ge\alpha_0, j\in\IZ\}$ and 
$K=\cup_{\beta\in\IZ} K_\beta$, where $K_\beta=K\cap \{a_{ij}:i\ge\alpha_0, j\ge \beta\}$ (=$K_{\alpha_0,\beta}$)  is in case (1).

\item If a subsemiheap $K$ of $E$ is in case (4), then 
$K\subset\{a_{ij}:i, j\in\IZ\}$ and 
$K=\cup_{\alpha\in\IZ} K_{(\alpha)}$, where $K_{(\alpha)}=K\cap \{a_{ij}:i\ge\alpha, j\in\IZ\}$  is in case (3).

Alternatively, if  a subsemiheap $K$ of $E$ is in case (4), then  
$K=\cup_{\beta\in\IZ} K^{(\beta)}$, where $K^{(\beta)}=K\cap \{a_{ij}:i\in\IZ, j\ge\beta\}$  is in case (2).

\end{itemize}
\end{remark}

\medskip
Therefore we shall concentrate only on case (1).
Suppose that $\alpha_0\ne-\infty$ and  $\beta_0\ne -\infty$. Then 
$K\subset K_{\alpha_0,\beta_0}=\{a_{pq}:p\ge\alpha_0,q\ge\beta_0\}.
$
We define three parameters as follows:
\[
\overline{\beta}=\sup\{\beta\in\IZ:a_{\alpha_0\beta}\in K\}
\]
\[
\overline{\alpha}=\sup\{\alpha\in\IZ:a_{\alpha\beta_0}\in K\}
\]
\[
\overline{\gamma}=\sup\{k\in\IN_0:a_{\alpha_0+k,\beta_0+k}\in K\}.
\]. 

We shall consider three primary cases:

\[
{\bf 1.}\ \overline{\beta}=\beta_0\quad\quad {\bf 2.}\  \beta_0<\overline{\beta}<\infty\quad\quad {\bf 3.}\ \overline{\beta}=\infty
\]

Each of the cases 1, 2, 3, consists of three further subcases. 

\[
{\bf 1.1}\ \overline{\beta}=\beta_0,\quad \overline{\alpha}=\alpha_0\quad\quad{\bf 1.2}\ \overline{\beta}=\beta_0, \quad \alpha_0<\overline{\alpha}<\infty,\quad {\bf 1.3}\ \overline{\beta}=\beta_0, \overline{\alpha}=\infty.
\]

\[
{\bf 2.1}\ \overline{\beta}=\beta_0<\infty, \overline{\alpha}=\alpha_0\quad\quad{\bf 2.2}\ \beta_0<\overline{\beta}<\infty, \alpha_0<\overline{\alpha}<\infty,\quad\quad {\bf 2.3}\ \beta_0<\overline{\beta}<\infty, \overline{\alpha}=\infty.
\]

\[
{\bf 3.1}\ \overline{\beta}=\infty, \overline{\alpha}=\alpha_0\quad\quad{\bf }\ \overline{\beta}=\infty, \alpha_0<\overline{\alpha}<\infty,\quad\quad {\bf 3.3}\  \overline{\beta}=\infty, \overline{\alpha}=\infty.
\]

Each of these nine cases consists of three further subcases.
 Thus, in order to account for the quadrant case (1), and hence the other three cases,  it will be necessary to consider 27 cases. We summarize the results in the  table on the next page. It is worthy to note that by Diagram~\ref{diagram:1022212}, if $ \overline{\beta}$ is finite, by which we mean, $\beta_0<\overline{\beta}<\infty$, then $a_{\alpha_0,\overline{\beta}}$ is the only point of $K$ of the form $a_{\alpha_0,\beta}$. A similar statement holds for $\overline{\alpha}$. Also, if $\overline{\alpha}=\alpha_0$, or if $\overline{\beta}=\beta_0$, then $a_{\alpha_0,\beta_0}\in K$.  Thus in cases 2.2, 2.3, 3.2, and 3.3, it is necessary to consider the two possibilities: $a_{\alpha_0,\beta_0}\in K$, and $a_{\alpha_0,\beta_0}\not\in K$.
 
We now proceed to analyze all 27 cases.

\medskip

\begin{lemma}\label{lem:1225211}
 In Case 1.1.1 ($\overline{\beta}=\beta_0,\  \overline{\alpha}=\alpha_0,\ \overline{\gamma}=0$),  we have  $K=\{a_{\alpha_0\beta_0}\}$.
\end{lemma}
\noindent{\it Proof.}
In this case, the diagram  is the following, where the bullet represents the element $a_{\alpha_0\beta_0}$, and the circles indicate that no element of $K$ occupies that position.   (Ignore, for the moment, the symbols $\blacksquare,\blacktriangle, \triangle$)

\medskip

\hspace{1.5in}$\bullet$\hspace{.22in}o\hspace{.22in}o\hspace{.22in}o\hspace{.22in}o\hspace{.22in}o
\vspace{.09in}

\hspace{1.5in}o\hspace{.2in}$\triangle$\hspace{.18in}$\blacksquare$
\vspace{.09in}

\hspace{1.5in}o\hspace{.2in}$\blacktriangle$ \hspace{.18in}o
\vspace{.09in}

\hspace{1.5in}o\hspace{.82in}o
\vspace{.09in}

\hspace{1.5in}o\hspace{1.12in}o
\vspace{.09in}





\hspace{1.5in}$\vdots$\hspace{1.4in}$\ddots$

Suppose that $a_{\alpha_0+2,\beta_0+1}$, denoted by $\blacktriangle$, belonged to $K$.  Then by diagram~\ref{diagram:1022213} applied to the points $a_{\alpha_0\beta_0}$ and $\blacktriangle$, the point $a_{\alpha_0+1,\beta_0+1}$, denoted by $\triangle$, would belong to $K$, a contradiction.  So $\blacktriangle$ does not belong to $K$. By the same argument, no element of $K$ resides in the second column of the diagram.

Suppose that $a_{\alpha_0+1,\beta_0+2}$, denoted by $\blacksquare$, belonged to $K$.  Then by diagram~\ref{diagram:1022211} applied to the points $a_{\alpha_0\beta_0}$ and $\blacksquare$, the point $a_{\alpha_0+1,\beta_0+1}$, denoted by $\triangle$, would belong to $K$, a contradiction. So $\blacksquare$ does not belong to $K$. By the same argument, no element of $K$ resides in the second row of the diagram.

Repetition of these two arguments shows that no element of $K$ resides in any column or row of the diagram, other than the first row and column, and therefore $K=\{a_{\alpha_0\beta_0}\}$ contains exactly one element.
\hfill$\Box$


{\begin{center}
{\bf Classification Scheme}
\end{center}

\begin{tabular}{| r|r |c |c |c |c|c|c|}\hline
&Case&subcase & $\overline{\beta}$ & $\overline{\alpha}$ &$\overline{\gamma}$ & exists? & result \\\hline\hline
&&1.1.1.     & $\beta_0$             & $\alpha_0$            &  0                             & yes & Lemma~\ref{lem:1225211}\\
&1.1&1.1.2.     & $\beta_0$             & $\alpha_0$            &   finite                            & yes & Proposition~\ref{prop:1225211}\\
&&1.1.3.     & $\beta_0$             & $\alpha_0$            &   $\infty$                            & yes & Proposition~\ref{prop:1225212}\\\hline

&&1.2.1.     & $\beta_0$             & finite           &   0                         & no &  Lemma~\ref{lem:1211211}  \\
1&1.2&1.2.2.     & $\beta_0$             & finite            &   finite                            & no &    Lemma~\ref{lem:1211211}\\
&&1.2.3.     & $\beta_0$             & finite            &   $\infty$                            & no& Lemma~\ref{lem:1211211}   \\\hline

&&1.3.1.     & $\beta_0$             & $\infty$            &  0                            & no &  Lemma~\ref{lem:1211211}  \\
&1.3&1.3.2.     & $\beta_0$             & $\infty$           &   finite                            & no &  Lemma~\ref{lem:1211211}  \\
&&1.3.3.     & $\beta_0$             & $\infty$            &   $\infty$                           & no & Lemma~\ref{lem:1211211}   \\\hline\hline

&&2.1.1.     & finite             & $\alpha_0$            &   0                           & no& Lemma~\ref{lem:1211212}   \\
&2.1&2.1.2.     & finite             & $\alpha_0$            &   finite                            & no &  Lemma~\ref{lem:1211212}  \\
&&2.1.3.     & finite             & $\alpha_0$            &   $\infty$                           & no &  Lemma~\ref{lem:1211212}  \\\hline

&&2.2.1.     & finite            & finite            &  0                          & no &  Lemma~\ref{lem:1211213}   \\
2&2.2&2.2.2.     & finite             & finite           &   finite                            & no &  Lemma~\ref{lem:1211213}   \\
&&2.2.3.     & finite             &  finite            &   $\infty$                              & no &  Lemma~\ref{lem:1211213}   \\\hline

&&2.3.1.     & finite             & $\infty$           &  0                          & no &  Lemma~\ref{lem:1211214}   \\
&2.3&2.3.2.     & finite             & $\infty$             &   finite                            & no& Lemma~\ref{lem:1211214}    \\
&&2.3.3.     & finite            & $\infty$             &    $\infty$                             & no &  Lemma~\ref{lem:1211214}   \\\hline\hline

&&3.1.1.     & $\infty$             & $\alpha_0$            &  0                          & no &  Lemma~\ref{lem:1211215}  \\
&3.1&3.1.2.     & $\infty$              & $\alpha_0$            &   finite                            & no  & Lemma~\ref{lem:1211215}   \\
&&3.1.3.     & $\infty$             & $\alpha_0$            &    $\infty$                             & no & Lemma~\ref{lem:1211215}   \\\hline

&&3.2.1.     & $\infty$             & finite            &   0                          &no  &  Lemma~\ref{lem:1211215}  \\
3&3.2&3.2.2.     & $\infty$              & finite            &   finite                            & no  & Lemma~\ref{lem:1211215}   \\
&&3.2.3.     & $\infty$              & finite           &   $\infty$                            & no  & Lemma~\ref{lem:1211215}   \\\hline

&&3.3.1.     & $\infty$              & $\infty$           &  0                           & no & Proposition~\ref{prop:1217211}  \\
&3.3&3.3.2.     & $\infty$             & $\infty$            &   finite                            & no  &  Proposition~\ref{prop:1217211}  \\
&&3.3.3.     & $\infty$              & $\infty$           &    $\infty$                            & yes & Proposition~\ref{prop:0102221}   \\\hline\hline

\end{tabular}

\medskip

\begin{proposition}\label{prop:1225211}
In Case 1.1.2 ($\overline{\beta}=\beta_0,\  \alpha_0=\overline{\alpha},\  0<\overline{\gamma}<\infty$), we have
$$
K=
\{a_{\alpha_0,\beta_0}\}\cup\{a_{\alpha_0+k_j,\beta_0+k_j}:1\le j\le n_0\},
$$ where  $1\le k_1<k_2<\cdots<k_{n_0}=\overline{\gamma}$ and  $a_{\alpha_0,\beta_0}$ and $a_{\alpha_0+k_j,\beta_0+k_j}$  are the only elements of $K$ that are in the diagonal.
\end{proposition}
\noindent{\it Proof.} In this case, in Diagram~\ref{diagram:0127221}, the bullets represent some of the elements of $K$ residing on the diagonal, the circles indicate that no element of $K$ occupies that position, and the dots represent both the finite number of points of $K$ on the diagonal together with some positions on the diagonal not containing points of $K$ (Ignore for the moment,  the symbols $\circledcirc, \boxdot $ which represent two elements of $K$ lying on the diagonal, and the symbols $\blacksquare,\square,\blacktriangle, \triangle$).
The symbol $\clubsuit$ represents the element $a_{\alpha_0+k_{n_0},\beta_0+k_{n_0}}$. We shall show that all off-diagonal positions are not occupied by elements of $K$, which means that $K=\{a_{\alpha_0,\beta_0}\}\cup\{a_{\alpha_0+k_j,\beta_0+k_j}:1\le j\le n_0\}$. 

Suppose that for $k_j\le\ell<k_{j+1}$, the point $a_{\alpha_0+k_{j+1}, \beta_0+\ell}$, denoted by $\blacksquare$ in Diagram~\ref{diagram:0127221}, belonged to $K$. Then   by Diagram~\ref{diagram:1022212}, starting with $\blacksquare$ and $a_{\alpha_0+k_{j+1},\beta_0+k_{j+1}}$, denoted by $\boxdot$ in Diagram~\ref{diagram:0127221}, shows that
\[
K\supset K_{\alpha_0+k_{j+1},\beta_0+\ell}^{k_{j+1}-\ell}=\{a_{\alpha_0+k_{j+1}+m(k_{j+1}-\ell),\beta_0+\ell+n(k_{j+1}-\ell)}:m,n\in\IN_0\}.
\]
Then choosing $m=n-1$, so that $$k_{j+1}+m(k_{j+1}-\ell)=\ell+n(k_{j+1}-\ell)$$ and letting $n\rightarrow\infty$ shows that
 there are infinitely many points of $K$ on the diagonal, a contradiction, so $\blacksquare\not\in K$. 
  The same argument applies to every point on each row determined by $\alpha_0+k_{j+1}$  to the left of $a_{\alpha_0+k_{j+1},\beta_0+k_{j+1}}$ for $0\le j\le n_0-1$.

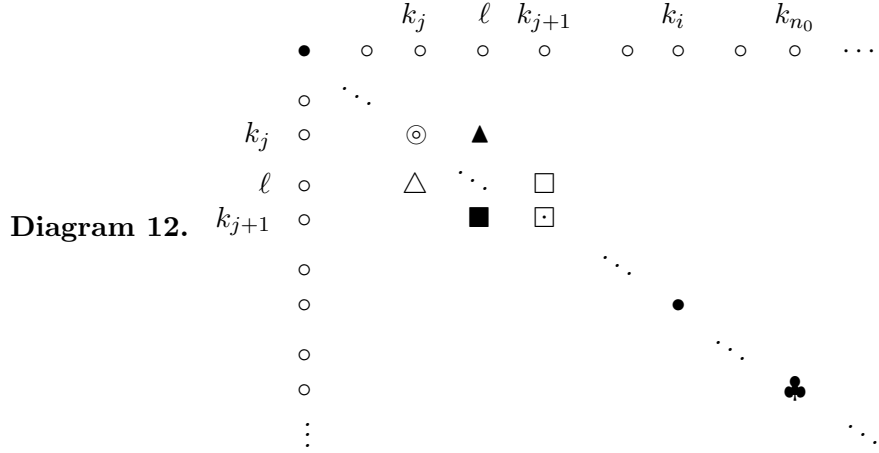
\begin{diagram}\label{diagram:0127221}
\begin{center}
\begin{tabular}{rrrrrcrrrcr}
&&&$k_j$&$\ell$&$k_{j+1}$&&$k_i$&&$k_{n_0}$&\\
&$\bullet$&$\circ$&$\circ$&$\circ$&$\circ$&$\circ$&$\circ$&$\circ$&$\circ$&$\cdots$\\
&$\circ$&$\ddots$&&&&&&&\\
$k_j$&$\circ$&&$\circledcirc$&$\blacktriangle$&&&&&\\
$\ell$&$\circ$&&$\triangle$&$\ddots$&$\square$&&&\\
$k_{j+1}$&  $\circ$& &&$\blacksquare$&$\boxdot$&&&&\\
&  $\circ$& &&&&$\ddots$&&&\\
&  $\circ$& &&&&&$\bullet$&&\\
&  $\circ$& &&&&&&$\ddots$&&\\
&  $\circ$& &&&&&&&$\clubsuit$&\\
&  $\vdots$& &&&&&&&&$\ddots$\\
\end{tabular}
\end{center}
\end{diagram}

\smallskip

\medskip

 Suppose   now that for $k_j<\ell\le k_{j+1}$, the point $a_{\alpha_0+k_{j}, \beta_0+\ell}$, denoted by $\blacktriangle$ in Diagram~\ref{diagram:0127221}, belonged to $K$. Then   by Diagram~\ref{diagram:1022212}, starting with $a_{\alpha_0+k_{j},\beta_0+k_{j}}$, denoted by $\circledcirc$ in Diagram~\ref{diagram:0127221}, and $\blacktriangle$, shows that
\[
K\supset K_{\alpha_0+k_{j},\beta_0+\ell}^{\ell-k_{j}}=
\{a_{\alpha_0+k_{j}+m(\ell-k_j),\beta_0+k_j+n(\ell-k_j)}:m,n\in\IN_0\}.
\]
Then choosing $m=n$,  and letting $n\rightarrow\infty$ show that
 there are infinitely many points of $K$ on the diagonal, a contradiction, so $\blacktriangle\not\in K$. 
  The same argument applies to every point on each row $\alpha_0+k_j$  to the right of $a_{\alpha_0+k_{j},\beta_0+k_{j}}$ for $0\le j\le n_0$. 
  
 Thus all rows containing an element of $K$ on the diagonal do not contain any other elements of $K$, as in the following diagram.  
 
 
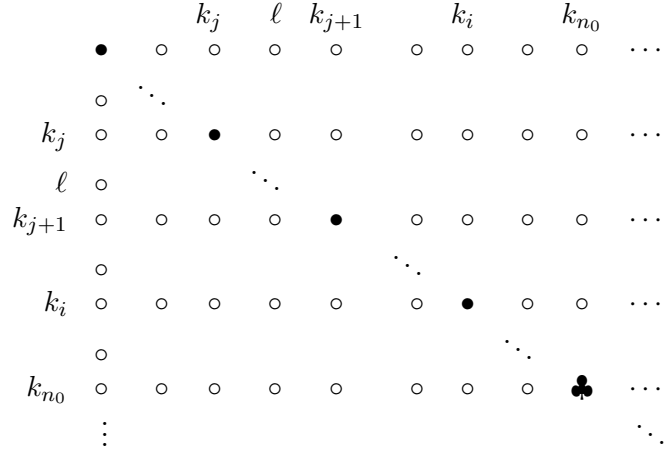
\begin{center}
\begin{tabular}{rrrrrcrrrcr}
&&&$k_j$&$\ell$&$k_{j+1}$&&$k_i$&&$k_{n_0}$&\\
&$\bullet$&$\circ$&$\circ$&$\circ$&$\circ$&$\circ$&$\circ$&$\circ$&$\circ$&$\cdots$\\
&$\circ$&$\ddots$&&&&&&&\\
$k_j$&$\circ$&$\circ$&$\bullet$&$\circ$&$\circ$&$\circ$&$\circ$&$\circ$&$\circ$&$\cdots$\\
$\ell$&$\circ$&&&$\ddots$&&&&\\
$k_{j+1}$&  $\circ$&$\circ$ &$\circ$&$\circ$&$\bullet$&$\circ$&$\circ$&$\circ$&  $\circ$&$\cdots$\\
&  $\circ$& &&&&$\ddots$&&&\\
$k_i$&  $\circ$& $\circ$&$\circ$&$\circ$&$\circ$&$\circ$&$\bullet$&$\circ$&$\circ$&$\cdots$\\
&  $\circ$& &&&&&&$\ddots$&&\\
$k_{n_0}$&  $\circ$&$\circ$ &$\circ$&$\circ$&$\circ$&$\circ$&$\circ$&$\circ$&$\clubsuit$&$\cdots$\\
&  $\vdots$& &&&&&&&&$\ddots$\\
\end{tabular}
\end{center}

\medskip

 A parallel argument, using Diagram~\ref{diagram:1022214} shows 
that all columns containing an element of $K$ on the diagonal do not contain any other elements of $K$. For completeness, we include the details.\smallskip

Suppose that for $k_j\le\ell<k_{j+1}$, the point $a_{\alpha_0+\ell, \beta_0+k_{j+1}}$, denoted by $\square$ in Diagram~\ref{diagram:0127221}, belonged to $K$. Then   by Diagram~\ref{diagram:1022214}, starting with $\square$ and $a_{\alpha_0+k_{j+1},\beta_0+k_{j+1}}$, denoted by $\boxdot$, shows that
\[
K\supset K_{\alpha_0+\ell,\beta_0+k_{j+1}}^{k_{j+1}-\ell}=
\{a_{\alpha_0+\ell+m(k_{j+1}-\ell),\beta_0+k_{j+1}+n(k_{j+1}-\ell)}:m,n\in\IN_0\}.
\]
Then choosing $m=n+1$, so that $$k_{j+1}+n(k_{j+1}-\ell)=\ell+m(k_{j+1}-\ell)$$ and letting $n\rightarrow\infty$ shows that
 there are infinitely many points of $K$ on the diagonal, a contradiction, so $\square\not\in K$. 
  The same argument applies to every point on each column determined by $\beta_0+k_{j+1}$  above  $a_{\alpha_0+k_{j+1},\beta_0+k_{j+1}}$ for $0\le j\le n_0-1$. \smallskip

 Suppose now that for $k_j\le\ell<k_{j+1}$, the point $a_{\alpha_0+\ell, \beta_0+k_j}$, denoted by $\triangle$ in Diagram~\ref{diagram:0127221}, belonged to $K$. Then   by Diagram~\ref{diagram:1022214}, starting with $\triangle$ and $a_{\alpha_0+k_j,\beta_0+k_j}$, denoted by $\circledcirc$, shows that
\[
K\supset K_{\alpha_0+k_{j},\beta_0+k_j}^{\ell-k_j}=
\{a_{\alpha_0+k_{j}+m(\ell-k_j),\beta_0+k_j+n(\ell-k_j)}:m,n\in\IN_0\},
\]
Then choosing $m=n$, so that $$k_j+n(\ell-k_j)=k_j+m(\ell-k_j)$$ and letting $n\rightarrow\infty$ shows that
 there are infinitely many points of $K$ on the diagonal, a contradiction, so $\triangle\not\in K$. 
  The same argument applies to every point on each column determined by $\beta_0+k_j$  above  $a_{\alpha_0+k_j,\beta_0+k_j}$ for $0\le j\le n_0$. \smallskip

To complete the proof, we now show that no point $a_{\alpha_0+m,\beta_0+n}$, with $m\ne n$ can belong to $K$. By what was just proved, it suffices to consider points which are not on a row or column containing a point of $K$, that is, $m\ne k_j$ for all $j$ and $n\ne k_\ell$ for all $\ell$.

We shall refer to the  diagram below, which depicts the eight possible locations for the element $a_{\alpha_0+m,\beta_0+n}$, reflecting the cases $m>n$ and $m<n$, namely,

\begin{enumerate}
\item $m> k_{n_0}\ge k_{\ell+1}>n> k_\ell$, denoted by $\blacksquare$
\item $k_{\ell+1}>m> k_\ell\ge k_{j+1}>n> k_j$, denoted by $\blacktriangle$
\item $k_{j+1}>m>n> k_j$, denoted by $\blacktriangledown$
\item $m>n> k_{n_0}$, denoted by $\blacktriangleleft$
\item $n> k_{n_0}\ge k_{\ell+1}>m> k_\ell$, denoted by $\square$
\item $k_{\ell+1}>n> k_\ell\ge k_{j+1}>m> k_j$, denoted by $\triangle$
\item $k_{j+1}>n>m> k_j$, denoted by $\triangledown$
\item $n>m> k_{n_0}$, denoted by $\triangleleft$
\end{enumerate}

Suppose first that $m>n$, for example case (2), $k_{\ell+1}>m_2>k_\ell\ge k_{j+1}>n_2>k_j$. 
 We consider the two points $a_{\alpha_0+k_j,\beta_0+k_j}$ and $\blacktriangle=a_{\alpha_0+m_2,\beta_0+n_2}$. These two points are vertices of a triangle with height $h=m_2-k_j$ greater than the base $b=n_2-k_j$, so $h-b=m_2-n_2$.  Then by Lemma~\ref{lem:0925212} (see Diagram~\ref{diagram:1022213}), the point $x_2=a_{\alpha_0+k_j+\beta_0+m_2-\alpha_0-k_j,\beta_o+m_2}=a_{\alpha_0+m_2,\beta_0+m_2}$ would belong to $K$, which is a contradiction since $k_\ell<m_2<k_{\ell+1}$.

  
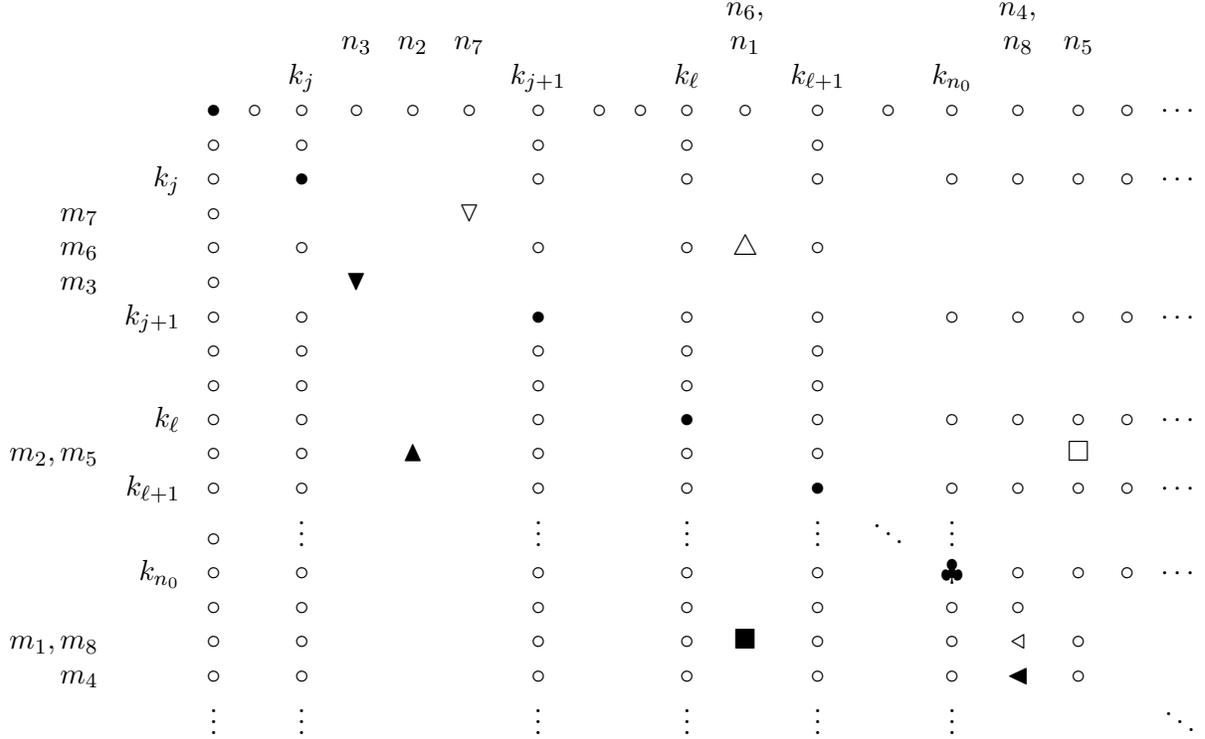
\begin{center}
\begin{tabular}{r r c c c c c c c c c  c c c c c c c c  c}
&& & &     & & &&        &  & &  & $n_6$, &  &   & &$n_4,$ &&&  \\ 
&& & &     &$n_3$ &$n_2$ &$n_7$ &        &  & &  & $n_1$ &  &   & &$n_8$& $n_5$&&  \\ 
&& & & $k_j$     & & & & $k_{j+1}$       &  & & $k_\ell$ &  & $k_{\ell+1}$ &   & $k_{n_0}$ && &&  \\ 
&& $\bullet$ &      $\circ$ & $\circ$ & $\circ$ &  $\circ$ & $\circ$ & $\circ$  & $\circ$ & $\circ$ & $\circ$  & $\circ$ & $\circ$ & $\circ$ & 
 $\circ$  & $\circ$ & $\circ$ & $\circ$ & $\cdots$  \\   
 &  & $\circ$  & & $\circ$  &&&& $\circ$ &&&   $\circ$        & & $\circ$  &  & &&&\\
& $k_j$ & $\circ$  & & $\bullet$ &&&& $\circ$  &&& $\circ$  & & $\circ$  && $\circ$ &$\circ$  & $\circ$ & $\circ$  &$\cdots$\\
$m_7$  & & $\circ$ & & & && $\triangledown$ &&&&&&&&&&&&\\
$m_6$ &  & $\circ$ & & $\circ$ & & & & $\circ$ &  & &  $\circ$ & $\triangle$& $\circ$ & &  &   & & & \\
$m_3$ &  & $\circ$ & & &  $\blacktriangledown$ & & & & & & & & & & & & & &  \\
&   $k_{j+1}$ & $\circ$ & & $\circ$ & & & & $\bullet$ & & & $\circ$ & & $\circ$ & & $\circ$ & $\circ$ & $\circ$ & $\circ$ & $\cdots$  \\  
 &  & $\circ$  & & $\circ$  &&&& $\circ$ &&&   $\circ$        & & $\circ$  &  & &&&\\
  &   & $\circ$  & & $\circ$  &&&& $\circ$ &&&   $\circ$        & & $\circ$  &  & &&&\\
 &  $k_\ell$  & $\circ$ & & $\circ$ & & & & $\circ$ & & & $\bullet$ & & $\circ$ & & $\circ$ & $\circ$ & $\circ$ & $\circ$& $\cdots$ \\
$m_2,m_5$&   & $\circ$  & & $\circ$  && $\blacktriangle$&& $\circ$ &&&   $\circ$        & & $\circ$  &  &  & & $\square$ &\\
 &  $k_{\ell+1}$  & $\circ$ & & $\circ$ & & & & $\circ$ & & & $\circ$ & & $\bullet$ & & $\circ$ & $\circ$ & $\circ$ & $\circ$& $\cdots$ \\
 &  & $\circ$ & & $\vdots$ & & & & $\vdots$& &  &   $\vdots$&  & $\vdots$ & $\ddots$& $\vdots$& & && \\
 &   $k_{n_0}$  & $\circ$ & & $\circ$ & & & & $\circ$ & & & $\circ$ & & $\circ$ & & $\clubsuit$ & $\circ$ & $\circ$ & $\circ$& $\cdots$ \\
   &  & $\circ$ & & $\circ$ & & & & $\circ$ & & & $\circ$ & & $\circ$ & & $\circ$ & $\circ$ &  & &  \\
$m_1,m_8$  &   & $\circ$ & & $\circ$ & & & & $\circ$ & & & $\circ$ &  $\blacksquare$ & $\circ$  && $\circ$ &$\triangleleft$ & $\circ$ & &  \\
$m_4$   & & $\circ$ & & $\circ$ & & & & $\circ$ & & & $\circ$ &  & $\circ$  && $\circ$ &$\blacktriangleleft$ & $\circ$ & &  \\
&& $\vdots$ & & $\vdots$  & & & & $\vdots$   & & & $\vdots$  & & $\vdots$ & & $\vdots$ &  & & & $\ddots$  \\
 \end{tabular}
\end{center}

Suppose next  that $m<n$, for example, case (6), $k_j<m_6<k_{j+1}\le k_\ell<n_6<k_{\ell+1}$
  We  again consider the two points $a_{\alpha_0+k_j,\beta_0+k_j}$ and $\triangle=a_{\alpha_0+m_6,\beta_0+n_6}$. These two points are vertices of a triangle with height $h=m_6-k_j$ less than the base $b=n_6-k_j$ and $b-h=n_6-m_6$.  Then by Lemma~\ref{lem:0925212} (see Diagram~\ref{diagram:1022211}), the point $x_1=a_{\alpha_0+k_j+\beta_0+n_6-\beta_0-k_j,\beta_0+n_6 }=a_{\alpha_0+n_6,\beta_0+n_6}$ would belong to $K$, which is a contradiction since $k_\ell<n_6<k_{\ell+1}$.

The same two-part argument works in all the other cases, more precisely, as follows:
\begin{itemize}
\item For case (3), $k_j<n_3<m_3<k_{j+1}$, Diagram 3 applied to $a_{\alpha_0+k_j,\beta_0+k_j}$ and $\blacktriangledown$ yields $x_2=a_{\alpha_0+n_3,\beta_0+n_3}$ 
\item For case (7), $k_j<m_7<n_7<k_{j+1}$ , Diagram 1 applied to $a_{\alpha_0+k_j,\beta_0+k_j}$ and $\triangledown$ yields $x_3=a_{\alpha_0+m_7,\beta_0+m_7}$
\item For case (4), $m_4> n_4> k_{n_0}$, Diagram 3 applied to $a_{\alpha_0+k_{n_0},\beta_0+k_{n_0}}$ and $\blacktriangleleft$ yields $x_2=a_{\alpha_0+n_4,\beta_0+n_4}$
\item For case (8),  $n_8>m_8> k_{n_0}$, 
 Diagram 1 applied to $a_{\alpha_0+k_{n_0},\beta_0+k_{n_0}}$ and $\triangleleft$ yields $x_3=a_{\alpha_0+m_8,\beta_0+m_8}$
\item For case (1), $m_1> k_{n_0}\ge k_{\ell+1}>n_1> k_\ell$, Diagram 3 applied to $a_{\alpha_0+k_\ell,\beta_0+k_\ell}$ and $\blacksquare$ yields $x_2=a_{\alpha_0+n_1,\beta_0+n_1}$
\item For case (5), $n_5> k_{n_0}\ge k_{\ell+1}>m_5> k_\ell$, Diagram 1 applied to $a_{\alpha_0+k_\ell,\beta_0+k_\ell}$ and $\square$ yields $x_3=a_{\alpha_0+m_5,\beta_0+m_5}$.\hfill$\Box$
\end{itemize}

This completes the analysis of case 1.1.2. 

 \begin{proposition}\label{prop:1225212}
In Case 1.1.3: ($\overline{\beta}=\beta_0,\ \overline{\alpha}=\alpha_0, \ \overline{\gamma}=\infty$), 
either
$K=\{a_{\alpha_0\beta_0}\}\cup\{a_{\alpha_0+k_j,\beta_0+k_j}:j\in\IN\}$, or there exist $\sigma,\ell_0\in\IN$, such that (see Examples~\ref{ex:1201211} and ~\ref{example:1220211})
 \begin{equation}\label{eq:0116223}
  K=\{a_{\alpha_0\beta_0}\}\cup\{a_{\alpha_0+k_j,\beta_0+k_j}:j=1,\ldots \ell_0-1\}\cup K_{\alpha_0+k_{\ell_0},\beta_0+k_{\ell_0}}^\sigma,\hbox{ or}
  \end{equation}
 \begin{equation}\label{eq:0116224}
  K=\{a_{\alpha_0\beta_0}\}\cup\{a_{\alpha_0+k_i,\beta_0+k_i}:1\le i<\ell_0\}\cup 
 \left(\bigcup_{i=\ell_0}^{j} K^{\sigma}_{\alpha_0+k_i,\beta_0+k_i}\right)
     \end{equation}
 where $1\le k_1<k_2<\cdots<k_n<\cdots <\infty$, $a_{\alpha_0,\beta_0}$ and $a_{\alpha_0+k_i,\beta_0+k_i}$  are the only elements of $K$ that are in the diagonal, and in
 (\ref{eq:0116224}),
  $\sigma=k_j-k_{\ell_0}$ where $k_j$ is such that no point 
  $a_{\alpha_0+k_{\ell_0},\beta_0+k_p}$ belongs to $K$ for $1\le p<j$.   If $\ell_0=1$, then the term $\{a_{\alpha_0+k_i,\beta_0+k_i}:1\le i<\ell_0\}$ is missing in (\ref{eq:0116223}) and (\ref{eq:0116224}).
\end{proposition}

\begin{proof}
In this case, the diagram  is the following, where the bullet represents the element $a_{\alpha_0\beta_0}$, the circles indicate that no element of $K$ occupies that position, and the dots indicate that infinitely many elements of $K$ reside in the diagonal.  

\medskip

$\bullet$\hspace{.28in}o\hspace{.25in}o\hspace{.25in}$\cdots$
\vspace{.09in}

o\hspace{.22in}$\ddots$
\vspace{.09in}

o\hspace{.52in}$\ddots$
\vspace{.09in}

$\vdots$\hspace{.82in}$\ddots$
\vspace{.09in}

We  consider a point $a_{\alpha_0+m,\beta_0+n}$, denoted by $\blacksquare$ in Diagram~\ref{diagram:0116221},  with $m\ne n$ and $m\ne k_j, n\ne k_j$ for every $j\ge 1$.

Suppose first that $m>n$, more precisely, $k_{j+1}>m>k_j\ge k_{\ell+1}>n>k_{\ell}$.  We consider the two points $a_{\alpha_0+k_\ell,\beta_0+k_\ell}$ (an element of $K$ on the diagonal), denoted by $\square$ in Diagram~\ref{diagram:0116221}, and $a_{\alpha_0+m,\beta_0+n}=\blacksquare$. 
These two points are vertices of a right triangle with height $h=m-k_\ell$ greater than the base $b=n-k_\ell$, so $h-b=m-n$.  Then by Lemma~\ref{lem:0925212} (see Diagram~\ref{diagram:1022213}), the point $x_2=a_{\alpha_0+k_\ell+\beta_0+n-\beta_0-k_\ell,\beta_o+m}=a_{\alpha_0+n,\beta_0+n}$, denoted by $\circledcirc$, would belong to $K$, which is a contradiction since $k_\ell<n<k_{\ell+1}$. Therefore all elements below the diagonal which are neither located on a row nor on a column  containing a diagonal point of  $K$, do not belong to $K$.

\medskip
\begin{diagram}\label{diagram:0116221}
\

\hspace{.78in}$k_\ell$ \hspace{.13in}\hspace{.06in}$k_{\ell+1}$\hspace{.5in}$k_j$ \hspace{.05in}\hspace{.15in}$k_{j+1}$ 

\hspace{.35in}$\bullet$\hspace{.15in}o\hspace{.15in}o\hspace{.15in}o\hspace{.15in}o\hspace{.15in}    o\hspace{.15in}o\hspace{.15in}o\hspace{.15in}o\hspace{.15in}o\hspace{.15in}o\hspace{.15in}o\hspace{.15in}$\cdots$

\hspace{.35in}o\hspace{.1in}$\ddots$
\vspace{.02in}

$k_\ell$\hspace{.17in} o \hspace{.32in}$\square$

$n(m)$\hspace{.04in}o\hspace{.25in}\hspace{.32in}$\circledcirc$\hspace{1.1in}$\blacktriangle$
\vspace{.06in}

$k_{\ell+1}$\hspace{.06in} o \hspace{.75in}$\bullet$

\hspace{.4in}o\hspace{1in}$\ddots$

\hspace{.42in}o\hspace{1.25in}$\ddots$ 
\vspace{.01in}

$k_j$\hspace{.28in}o\hspace{1.55in}$\bullet$

$m(n)$\hspace{.1in}o\hspace{.55in}$\blacksquare$\hspace{1in} $\circ$
\vspace{.03in}

$k_{j+1}$\hspace{.15in}o\hspace{1.95in}$\bullet$

\hspace{.42in}o\hspace{2.2in}$\ddots$ 
\vspace{.01in}


\end{diagram}
\medskip

Suppose next  that $m<n$, more precisely, $k_\ell<m<k_{\ell+1}\le k_j<n<k_{j+1}$
  We  again consider the two points $a_{\alpha_0+k_\ell,\beta_0+k_\ell}$, denoted by $\square$ in Diagram~\ref{diagram:0116221}, and $a_{\alpha_0+m,\beta_0+n}$, denoted by $\blacktriangle$ in  Diagram~\ref{diagram:0116221}. These two points are vertices of a right triangle with height $h=m-k_\ell$ smaller than the base $b=n-k_\ell$ and $b-h=n-m$.  Then by Lemma~\ref{lem:0925212} (see Diagram~\ref{diagram:1022211}), the point $x_1=a_{\alpha_0+k_l+\beta_0+n-\beta_0-k_\ell,\beta_0+n }=a_{\alpha_0+n,\beta_0+n}$, denoted by $\circledcirc$, would belong to $K$, which is a contradiction since $k_j<n<k_{j+1}$.

  The same two part argument works in the cases $k_j<m<n<k_{j+1}$ and $k_j<n<m<k_{j+1}$. Therefore all elements above or below the diagonal which are not located on a row or on a column  containing a diagonal point of  $K$, do not belong to $K$.

 We now have Diagram~\ref{diagram:0116222}.  Thus the only  off-diagonal points that can possibly belong to $K$ are those that are located either on a row containing a diagonal point of $K$, indicated by $\cdots$, or on a column
containing a diagonal point of $K$, indicated by vertical dots. 

 We next show that points which lie on a row or column, but not both, cannot belong to $K$.  
  
  Consider first, for any row determined by $k_j$ and any $n$ with $n \ne k_m$ for all $m$, the element $a_{\alpha_0+k_j,\beta_0+n}$, denoted by $\blacksquare$ in Diagram~\ref{diagram:0116223}, and suppose it belonged to $K$.  Then by Lemma~\ref{lem:0925211}(iii), 
  \[
  a_{\alpha_0\beta_0}a^*_{\alpha_0+k_j,\beta_0+n}a_{\alpha_0+k_j,\beta_0+n}=a_{\alpha_0+n,\beta_0+n}\hbox{ (denoted by $\square$)}
  \]
  would belong to $K$, a contradiction since $n \ne k_m$ for every $m$.
  
Next,   for any column determined by $k_j$ and any $n$ with $n \ne k_m$ for all $m$, consider the element $a_{\alpha_0+n,\beta_0+k_j}$, denoted by $\blacktriangle$ in Diagram~\ref{diagram:0116223}, and suppose it belonged to $K$.  Then by Lemma~\ref{lem:0925211}(i), 
  \[
  a_{\alpha_0+n,\beta_0+k_j}a^*_{\alpha_0+n,\beta_0+k_j}a_{\alpha_0\beta_0}=a_{\alpha_0+n,\beta_0+n} \hbox{ (denoted by $\square$)}
 \]
  would belong to $K$, a contradiction since $n \ne k_m$ for every $m$.
  \medskip
  
  We now have Diagram~\ref{diagram:0116224}.
  Thus the only  off-diagonal points that can possibly belong to $K$ are those, denoted by $\blacksquare$, that are located simultaneously on a row containing a diagonal point of $K$ and a column
containing a diagonal point of $K$.

  If  none of the  off-diagonal elements $\blacksquare$ belong to $K$, then obviously $K=\{a_{\alpha_0\beta_0}\}\cup\{a_{\alpha_0+k_j,\beta_0+k_j}:j\in\IN\}$.  We next consider the case that some of the elements $\blacksquare$ in Diagram~\ref{diagram:0116224} belong to $K$.

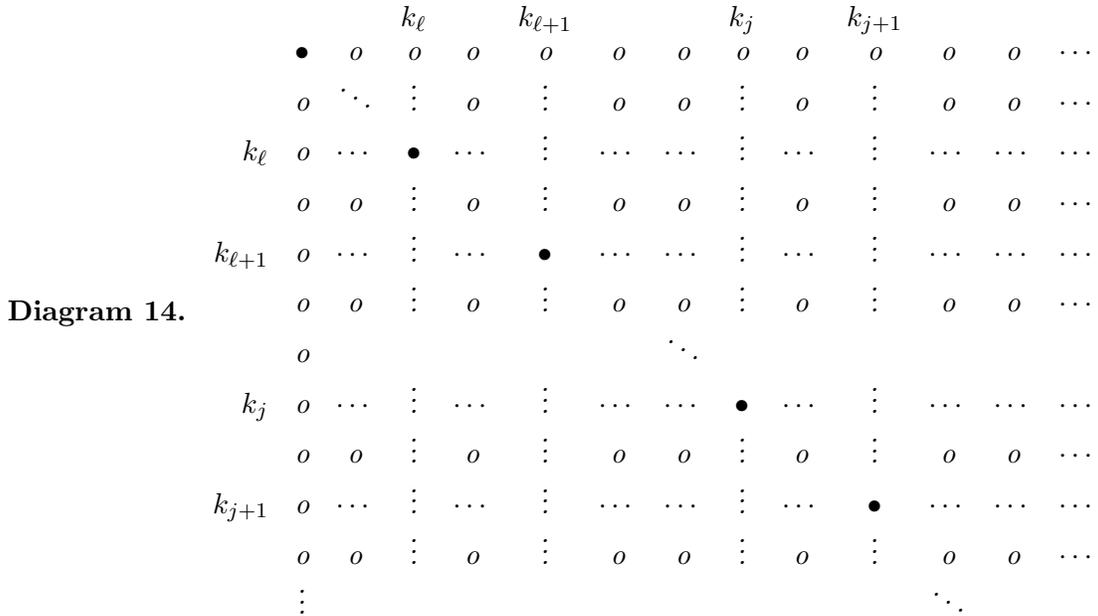
\begin{diagram}\label{diagram:0116222}

\begin{center}
\begin{tabular}{r ccccccccccccc}
& & & $k_\ell$ &  & $k_{\ell+1}$ & & & $k_j$ &  & $k_{j+1}$ & & & \\ 
 & $\bullet$ & o & o & o & o & o & o & o & o & o & o & o & $\cdots$ \\   
& o & $\ddots$  & $\vdots$ & o & $\vdots$ & o & o & $\vdots$ & o & $\vdots$ & o &  o & $\cdots$\\ 
$k_\ell$ & o & $\cdots$ & $\bullet$ & $\cdots$ & $\vdots$ & $\cdots$ & $\cdots$ & $\vdots$ & $\cdots$ & $\vdots$ & $\cdots$ & $\cdots$ & $\cdots$ \\ 
 & o & o  & $\vdots$ &  o & $\vdots$ & o & o & $\vdots$ & o & $\vdots$ & o &  o& $\cdots$\\  
$k_{\ell+1}$ & o & $\cdots$ & $\vdots$ & $\cdots$ & $\bullet$ & $\cdots$ & $\cdots$ & $\vdots$ & $\cdots$ & $\vdots$ & $\cdots$ & $\cdots$ & $\cdots$ \\ 
& o & o  & $\vdots$ & o  & $\vdots$ & o & o & $\vdots$ & o & $\vdots$ & o &  o& $\cdots$\\ 
& o & & & & &  & $\ddots$  & & & & & & \\ 
$k_j$ & o & $\cdots$ & $\vdots$ & $\cdots$ & $\vdots$ & $\cdots$ & $\cdots$ & $\bullet$ & $\cdots$ & $\vdots$ & $\cdots$ & $\cdots$ & $\cdots$ \\ 
 & o & o  & $\vdots$ & o  & $\vdots$ & o & o & $\vdots$ & o & $\vdots$ & o &  o& $\cdots$\\ 
$k_{j+1}$ & o & $\cdots$ & $\vdots$ & $\cdots$ & $\vdots$ & $\cdots$ & $\cdots$ & $\vdots$ & $\cdots$ & $\bullet$ & $\cdots$ & $\cdots$ & $\cdots$ \\ 
& o & o  & $\vdots$ & o  & $\vdots$ & o & o & $\vdots$ & o & $\vdots$ & o &  o& $\cdots$\\ 
& $\vdots$ & & & & & & & & & &  $\ddots$ &   &
\end{tabular}
\end{center}
  \end{diagram}
  
  A row determined by $k_j$ for which no element other than $a_{\alpha_0+k_j,\beta_0+k_j}$  belongs to $K$, that is,  $a_{\alpha_0+k_j,\beta_0+k_j}\in K$, and  $a_{\alpha_0+k_j,\beta_0+k_p}\not\in K$ for all $p\in\IN-\{k_j\}$, will be called a {\it null } row.  More precisely,  a {\it right-null} (respectively {\it left-null}) row determined by $k_j$ is one that satisfies
 $a_{\alpha_0+k_j,\beta_0+k_j}\in K$, and $a_{\alpha_0+k_j,\beta_0+k_p}\not\in K$ for all $p> k_j$  (respectively $p<k_j$). The row determined by $\alpha_0$ is a null-row.

  
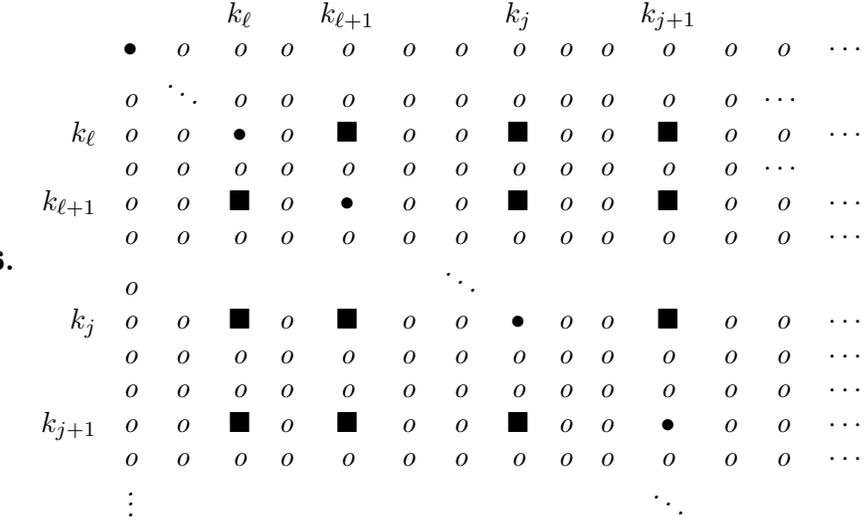
\begin{diagram}\label{diagram:0116223}
  \begin{center}
\begin{tabular}{r ccccccccccccc}
& & & $k_\ell$ & $n$ & $k_{\ell+1}$ & & & $k_j$ &  & $k_{j+1}$ & & & \\ 
 & $\bullet$ & o & o & o & o & o & o & o & o & o & o & o & $\cdots$ \\   
& o & $\ddots$  & $\vdots$ & o & $\vdots$ & o & o & $\vdots$ & o & $\vdots$ & o &  o & $\cdots$\\ 
$k_\ell$ & o & $\cdots$ & $\bullet$ & $\cdots$ & $\vdots$ & $\cdots$ & $\cdots$ & $\vdots$ & $\cdots$ & $\vdots$ & $\cdots$ & $\cdots$ & $\cdots$ \\ 
$n$ & o & o  & $\vdots$ &  $\square$ & $\vdots$ & o & o & $\blacktriangle$ & o & $\vdots$ & o &  o& $\cdots$\\  
$k_{\ell+1}$ & o & $\cdots$ & $\vdots$ & $\cdots$ & $\bullet$ & $\cdots$ & $\cdots$ & $\vdots$ & $\cdots$ & $\vdots$ & $\cdots$ & $\cdots$ & $\cdots$ \\ 
& o & o  & $\vdots$ & o  & $\vdots$ & o & o & $\vdots$ & o & $\vdots$ & o &  o& $\cdots$\\ 
& o & & & & &  & $\ddots$  & & & & & & \\ 
$k_j$ & o & $\cdots$ & $\vdots$ & $\blacksquare$ & $\vdots$ & $\cdots$ & $\cdots$ & $\bullet$ & $\cdots$ & $\vdots$ & $\cdots$ & $\cdots$ & $\cdots$ \\ 
 & o & o  & $\vdots$ &$\circ$  & $\vdots$ & o & o & $\vdots$ & $\circ$ & $\vdots$ & o &  o& $\cdots$\\ 
$k_{j+1}$ & o & $\cdots$ & $\vdots$ & $\cdots$ & $\vdots$ & $\cdots$ & $\cdots$ & $\vdots$ & $\cdots$ & $\bullet$ & $\cdots$ & $\cdots$ & $\cdots$ \\ 
& o & o  & $\vdots$ & o  & $\vdots$ & o & o & $\vdots$ & o & $\vdots$ & o &  o& $\cdots$\\ 
& $\vdots$ & & & & & & & & & &  $\ddots$ &   &
\end{tabular}
\end{center}
\end{diagram}

    \begin{diagram}\label{diagram:0116224}
  \begin{center}
\begin{tabular}{r cccccccccccccc}
& & & $k_\ell$ &  & $k_{\ell+1}$ & & & $k_j$ &  && $k_{j+1}$ & & & \\ 
 & $\bullet$ & o & o & o & o & o & o & o & o & o & o & o & o & $\cdots$ \\   
& o & $\ddots$  & o& o  & o & o & o & o & o & o & o &  o& $\cdots$\\  
$k_\ell$ & o & o & $\bullet$ & o & $\blacksquare$ & o & o & $\blacksquare$ & o & o&$\blacksquare$  & o& o & $\cdots$ \\ 
& o & o  & o& o  & o & o & o & o & o & o & o &  o& $\cdots$\\   
$k_{\ell+1}$ & o & o & $\blacksquare$ & o & $\bullet$ & o & o & $\blacksquare$ & o & o &$\blacksquare$ & o & o & $\cdots$\\ 
& o & o  & o& o  & o & o & o & o & o & o & o & o &  o& $\cdots$\\ 
& o & & & & &  & $\ddots$  & & & & & & & \\ 
$k_j$ & o & o & $\blacksquare$ & o & $\blacksquare$ & o & o & $\bullet$ & o & o & $\blacksquare$ & o & o & $\cdots$\\ 
 & o & o  & o& o  & o & o & o & o & o & o & o & o &  o& $\cdots$\\ 
 & o & o  & o& o  & o & o & o & o & o & o & o & o &  o& $\cdots$\\ 
$k_{j+1}$ & o & o & $\blacksquare$ & o & $\blacksquare$ & o & o & $\blacksquare$ & o & o & $\bullet$ & o & o & $\cdots$\\
& o & o  & o& o  & o &o & o & o & o & o & o & o &  o& $\cdots$\\ 
& $\vdots$ & & & & & & & & & &  $\ddots$ &   &&
\end{tabular}
\end{center}
\end{diagram}  
 
  Similarly, a row determined by $k_j$ which contains an element of $K$ other than $a_{\alpha_0+k_j,\beta_0+k_j}$, that is, there exists $\ell>j$   (respectively $\ell<j$), such that $a_{\alpha_0+k_j,\beta_0+\ell}\in K$,  will be called a {\it right-ample} (respectively {\it left-ample}) row.  A row which is either left-ample or right-ample (or both), will be called simply {\it ample}.   By Diagram~\ref{diagram:1022212}, a left-ample row is also right-ample, but not conversely (see the sentence following Lemma~\ref{lem:0925211}). For the same reason, a right-null row is also left-null.
As noted above, if all rows of $K$ are null, then $K=\{a_{\alpha_0\beta_0}\}\cup\{a_{\alpha_0+k_j,\beta_0+k_j}:j\in\IN\}$. Thus we have the following lemma.

\begin{lemma}
If all the rows of $K$  are right-null, then $K=\{a_{\alpha_0\beta_0}\}\cup\{a_{\alpha_0+k_j,\beta_0+k_j}:j\in\IN\}.$
\end{lemma}
 
 \begin{lemma}
 All right-null rows lie above all ample rows. Hence, if there is at least one ample row, then there are only finitely many right-null rows.
 \end{lemma}
  \begin{proof}
  Suppose $k_j$ determines a right-null row, and $k_\ell$ determines an ample row, which we may assume to be  right-ample, say containing the element $a_{\alpha_0+k_\ell,\beta_0+k_m}$,  and suppose by way of contradiction that $\ell<j$. Since  $\ell<j<m$,  the diagram is the following, where $\blacksquare$  denotes the element $a_{\alpha_0+k_j,\beta_0+k_m}$.

\begin{center}
\begin{tabular}{lccccccccc}
&&$k_\ell$&&&$k_j$&&$k_m$&&\\
&$\cdots$&$\cdots$&$\cdots$&$\cdots$&$\cdots$&$\cdots$&$\cdots$&$\cdots$&$\cdots$\\
$k_\ell$&$\circ$&$\bullet$&$\circ$&$\circ$&$\circ$&$\circ$&$\blacksquare$&$\circ$&$\cdots$\\
&$\cdots$&$\cdots$&$\cdots$&$\cdots$&$\cdots$&$\cdots$&$\cdots$&$\cdots$&$\cdots$\\
$k_j$&$\circ$&$\cdots$&$\circ$&$\cdots$&$\bullet$&$\circ$&$\circ$&$\circ$&$\cdots$\\
&$\cdots$&$\cdots$&$\cdots$&$\cdots$&$\cdots$&$\cdots$&$\cdots$&$\cdots$&$\cdots$\\
$k_m$&$\cdots$&$\cdots$&$\circ$&$\cdots$&$\cdots$&$\circ$&$\bullet$&$\circ$&$\cdots$\\

\end{tabular}
\end{center}

 The two points $\blacksquare$ and $a_{\alpha_0+k_j,\beta_0+k_j}$ are vertices of a triangle with base  $b=k_m-k_j$ and height $h=k_j-k_\ell$.  Then by Diagrams~\ref{diagram:1027211}, \ref{diagram:1027213}, or \ref{diagram:1027215}, depending on the relative sizes of $b$ and $h$, the point $x_3=a_{\alpha_0+k_j,\beta_0+k_m+k_j-k_\ell}$ would belong to $K$, which is a contradiction since $k_m+k_j-k_\ell> k_j$.
  \end{proof}

  Let $\ell_0\in\IN$ be such that the first ample row is determined by $k_{\ell_0}$, and assume without loss of generality, that this row is right-ample.  Assume also, temporarily, that $\ell_0>1$. The rows lying above the row determined by $\alpha_0+k_{\ell_0}$ do not contain any elements of $K$ above the diagonal. It follows from Diagram~\ref{diagram:1022214} that the columns lying to the left of the column determined by $\beta_0+k_{\ell_0}$ do not contain any elements of $K$ below the diagonal.  By Diagram~\ref{diagram:1022212}, $K$ contains infinitely many elements to the right of $a_{\alpha_0+k_{\ell_0},\beta_0+k_{\ell_0}}$, and then by 
  Diagrams~\ref{diagram:1022212} and ~\ref{diagram:1022214}, $K$ contains infinitely many elements below $a_{\alpha_0+k_{\ell_0},\beta_0+k_{\ell_0}}$. Since $\overline{\gamma}=\infty$, the subsemiheap $K\cap K_{\alpha_0+k_{\ell_0},\beta_0+k_{\ell_0}}$ falls into subcase 3.3.3 below (see Propositions~\ref{prop:1213211} and ~\ref{prop:1211211}, and Diagram~\ref{diagram:0116225}), and therefore
  \begin{equation}\label{eq:0116221}
  K=\{a_{\alpha_0\beta_0}\}\cup\{a_{\alpha_0+k_i,\beta_0+k_i}:1\le i<\ell_0\}\cup K_{\alpha_0+k_{\ell_0},\beta_0+k_{\ell_0}}^{k_j-k_{\ell_0}},
  \end{equation}
 or
 \begin{equation}\label{eq:0116222}
  K=\{a_{\alpha_0\beta_0}\}\cup\{a_{\alpha_0+k_i,\beta_0+k_i}:1\le i<\ell_0\}\cup 
 \left(\bigcup_{i=\ell_0}^{j} K^{k_j-k_{\ell_0}}_{\alpha_0+k_i,\beta_0+k_i}\right)
    \end{equation}
  where $k_j$ is such that no point 
  $a_{\alpha_0+k_{\ell_0},\beta_0+k_p}$ belongs to $K$ for $1\le p<j$.   If $\ell_0=1$, then the term $\{a_{\alpha_0+k_i,\beta_0+k_i}:1\le i<\ell_0\}$ is missing in (\ref{eq:0116221}) and (\ref{eq:0116222}).  (In Diagram~\ref{diagram:0116225}, $\sigma=k_{\ell_0+1}+(k_j-k_{\ell_0})$). This completes the proof of Proposition~\ref{prop:1225212}. \end{proof}

   This completes the analysis of case 1.1.3 and hence of case 1.1.   
 
   \begin{lemma}\label{lem:1211211}
Cases 1.2 and 1.3 do not occur.
\end{lemma}
 
 \begin{proof}
 In each of the six subcases  of cases 1.2 and 1.3, the diagram is the following:

\begin{center}
\begin{tabular}{r c c c c c c c}
&$\beta_0$&&&&&&\\
$\alpha_0$&$\bullet$&$\circ$&$\circ$&$\circ$&$\circ$&$\circ$&$\cdots$\\
&$\circ$&$\ddots$&&&&&\\
&$\circ$&&$\ddots$&&&&\\
&$\circ$&&&$\ddots$&&&\\
$r$&$\bullet$&&&&$\ddots$&&\\
&$\vdots$&&&&&$\ddots$&\\
&$\vdots$&&&&&&$\ddots$\\
\end{tabular}
\end{center}
 
 Then applying Diagram~\ref{diagram:1022214} to the points $a_{\alpha_0\beta_0}$ and $a_{\alpha_0+r,\beta_0}$ shows that $x_1=a_{\alpha_0,\beta_0+r}\in K$, a contradiction.
 \end{proof} 
  
  \begin{diagram}\label{diagram:0116225}

 \begin{center}
 \begin{tabular}{rcccccccccccccccccc}
 &$\beta_0$&$\cdots$&$k_1$&$\cdots$&$k_2$&$\cdots$&$k_{\ell_0-1}$&$\cdots$&$k_{\ell_0}$&$\cdots$&$k_{\ell_0+1}$&$\cdots$&$k_j$&$\cdots$&$\sigma$&$\cdots$&$\cdots$\\
 
$\alpha_0$ &$\bullet$&$\circ$&$\circ$&$\circ$&$\circ$&$\circ$&$\circ$&$\circ$&$\circ$&$\circ$&$\circ$&$\circ$&$\circ$&$\circ$&$\circ$&$\circ$&$\cdots$&\\

&$\circ$&$\circ$&$\circ$&$\circ$&$\circ$&$\circ$&$\circ$&$\circ$&$\circ$&$\circ$&$\circ$&$\circ$&$\circ$&$\circ$&$\circ$&$\circ$&$\cdots$&\\

$k_1$&$\circ$&$\circ$&$\bullet$&$\circ$&$\circ$&$\circ$&$\circ$&$\circ$&$\circ$&$\circ$&$\circ$&$\circ$&$\circ$&$\circ$&$\circ$&$\circ$&$\cdots$&\\

&$\circ$&$\circ$&$\circ$&$\circ$&$\circ$&$\circ$&$\circ$&$\circ$&$\circ$&$\circ$&$\circ$&$\circ$&$\circ$&$\circ$&$\circ$&$\circ$&$\cdots$&\\

$k_2$&$\circ$&$\circ$&$\circ$&$\circ$&$\bullet$&$\circ$&$\circ$&$\circ$&$\circ$&$\circ$&$\circ$&$\circ$&$\circ$&$\circ$&$\circ$&$\circ$&$\cdots$&\\

$\vdots$&$\circ$&$\circ$&$\circ$&$\circ$&$\circ$&$\circ$&$\circ$&$\circ$&$\circ$&$\circ$&$\circ$&$\circ$&$\circ$&$\circ$&$\circ$&$\circ$&$\cdots$&\\

$k_{\ell_0-1}$&$\circ$&$\circ$&$\circ$&$\circ$&$\circ$&$\circ$&$\bullet$&$\circ$&$\circ$&$\circ$&$\circ$&$\circ$&$\circ$&$\circ$&$\circ$&$\cdots$&\\

$\vdots$&$\circ$&$\circ$&$\circ$&$\circ$&$\circ$&$\circ$&$\circ$&$\circ$&$\circ$&$\circ$&$\circ$&$\circ$&$\circ$&$\circ$&$\circ$&$\circ$&$\cdots$\\

$k_{\ell_0}$&$\circ$&$\circ$&$\circ$&$\circ$&$\circ$&$\circ$&$\circ$&$\circ$&$\bullet$&$\circ$&$\circ$&$\circ$&$\bullet$&$\circ$&$\circ$&$\cdots$&\\

$\vdots$&$\circ$&$\circ$&$\circ$&$\circ$&$\circ$&$\circ$&$\circ$&$\circ$&$\circ$&$\circ$&$\circ$&$\circ$&$\circ$&$\circ$&$\circ$&$\circ$&$\cdots$\\

$k_{\ell_0+1}$&$\circ$&$\circ$&$\circ$&$\circ$&$\circ$&$\circ$&$\circ$&$\circ$&$\circ$&$\circ$&$\bullet$&$\circ$&$\circ$&$\circ$&$\bullet$&$\circ$&$\cdots$\\

$\vdots$&$\circ$&$\circ$&$\circ$&$\circ$&$\circ$&$\circ$&$\circ$&$\circ$&$\circ$&$\circ$&$\circ$&$\circ$&$\circ$&$\circ$&$\circ$&$\circ$&$\cdots$\\

$k_j$&$\circ$&$\circ$&$\circ$&$\circ$&$\circ$&$\circ$&$\circ$&$\circ$&$\bullet$&$\circ$&$\circ$&$\circ$&$\bullet$&$\circ$&$\circ$&$\circ$&$\cdots$&\\

$\vdots$&$\circ$&$\circ$&$\circ$&$\circ$&$\circ$&$\circ$&$\circ$&$\circ$&$\circ$&$\circ$&$\circ$&$\circ$&$\circ$&$\circ$&$\circ$&$\circ$&$\cdots$&\\

$\sigma$&$\circ$&$\circ$&$\circ$&$\circ$&$\circ$&$\circ$&$\circ$&$\circ$&$\circ$&$\circ$&$\bullet$&$\circ$&$\circ$&$\circ$&$\bullet$&$\circ$&$\cdots$&\\

$\vdots$&$\circ$&$\circ$&$\circ$&$\circ$&$\circ$&$\circ$&$\circ$&$\circ$&$\circ$&$\circ$&$\circ$&$\circ$&$\circ$&$\circ$&$\circ$&$\circ$&$\cdots$&\\

$\vdots$&&&$\vdots$&&$\vdots$&&$\vdots$&&$\vdots$&&$\vdots$&&$\vdots$&&$\vdots$&&&\\


 \end{tabular}
 \end{center} 
  
 \end{diagram}

\begin{lemma}\label{lem:1211212}
Case 2.1 does not occur.
\end{lemma}
 
 \begin{proof}
 In each of the three subcases  of case 2.1, the diagram is the following:

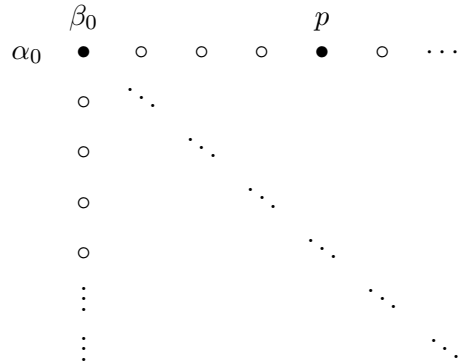
\begin{center}
\begin{tabular}{r c c c c c c c}
&$\beta_0$&&&&$p$&&\\
$\alpha_0$&$\bullet$&$\circ$&$\circ$&$\circ$&$\bullet$&$\circ$&$\cdots$\\
&$\circ$&$\ddots$&&&&&\\
&$\circ$&&$\ddots$&&&&\\
&$\circ$&&&$\ddots$&&&\\
&$\circ$&&&&$\ddots$&&\\
&$\vdots$&&&&&$\ddots$&\\
&$\vdots$&&&&&&$\ddots$\\
\end{tabular}
\end{center}
 
 Then applying Diagram~\ref{diagram:1022212} to the points $a_{\alpha_0\beta_0}$ and $a_{\alpha_0,\beta_0+p}$ shows that $x_1=a_{\alpha_0+p,\beta_0}\in K$, a contradiction.
 \end{proof}

  \begin{lemma}\label{lem:1211213}
Case 2.2 does not occur.
\end{lemma}
\begin{proof}
 Note first that by Diagrams~\ref{diagram:1022212} and ~\ref{diagram:1022214}, only one element of $K$ can reside on the row determined  by $\alpha_0$ or on the column determined by $\beta_0$. Thus, the diagram for this case (if $r<p$) is the following.
 
 
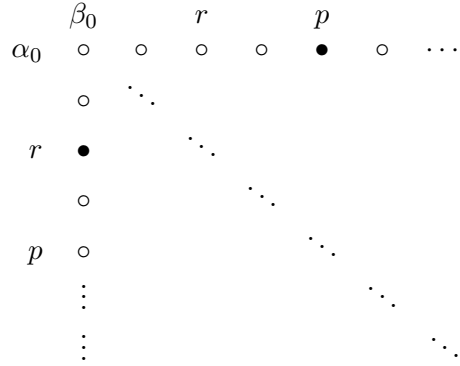
\begin{center}
\begin{tabular}{r c c c c c c c}
&$\beta_0$&&$r$&&$p$&&\\
$\alpha_0$&$\circ$&$\circ$&$\circ$&$\circ$&$\bullet$&$\circ$&$\cdots$\\
&$\circ$&$\ddots$&&&&&\\
$r$&$\bullet$&&$\ddots$&&&&\\
&$\circ$&&&$\ddots$&&&\\
$p$&$\circ$&&&&$\ddots$&&\\
&$\vdots$&&&&&$\ddots$&\\
&$\vdots$&&&&&&$\ddots$\\
\end{tabular}
\end{center}

Then by Diagrams~~\ref{diagram:1027211} and ~\ref{diagram:1027213} (depending on whether $r<p$ or $p\ge r$, $K$ would contain elements  $a_{\alpha_0,\beta_0+\ell}$ with $\ell>p$, a contradiction.
\end{proof}

  \begin{lemma}\label{lem:1211214}
Case 2.3 does not occur, hence case 2 does not occur.
\end{lemma}
\begin{proof}
In Case 2.3.1 ($ \beta_0<\overline{\beta}<\infty, \overline{\alpha}=\infty,\overline{\gamma}=0$),
 note that $a_{\alpha_0,\beta_0+p}$ is the only point of $K$ on the row determined by $\alpha_0$.  From the following diagram we see that if $p<r$, we get a contradiction using Diagram~\ref{diagram:1027213}, and if $r<p$ we get a contradiction using Diagram~\ref{diagram:1027211}, whereas if $p=r$, we get a contradiction using Diagram~\ref{diagram:1027215}.

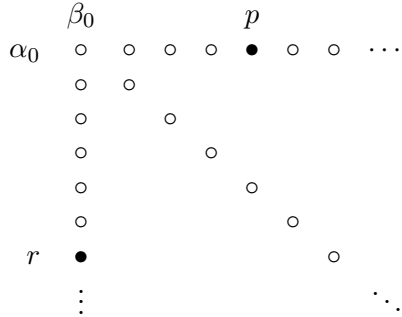
 \begin{center}
\begin{tabular}{r c c c c c c c c c}
&$\beta_0$&&&&$p$&&&&\\
$\alpha_0$&$\circ$&$\circ$&$\circ$&$\circ$&$\bullet$&$\circ$&$\circ$&$\cdots$&\\
&$\circ$&$\circ$&&&&&&&\\
&$\circ$&&$\circ$&&&&&&\\
&$\circ$&&&$\circ$&&&&&\\
&$\circ$&&&&$\circ$&&&&\\
&$\circ$&&&&&$\circ$&&&\\
$r$&$\bullet$&&&&&&$\circ$&&\\
&$\vdots$&&&&&&&$\ddots$&\\
\end{tabular}
\end{center}

The same proof applies to cases  2.3.2 ($\beta_0<\overline{\beta}<\infty, \overline{\alpha}=\infty,0<\overline{\gamma}<\infty$)  and  2.3.3 ($ \beta_0<\overline{\beta}<\infty, \overline{\alpha}=\infty, \overline{\gamma}=\infty$).\end{proof}

\begin{lemma}\label{lem:1211215}
 Cases 3.1 and  3.2 do not occur.
 \end{lemma}
 \begin{proof}
Cases 3.1.1, 3.1.2, and 3.1.3 do not occur by Diagram~\ref{diagram:1022212}. 
 Cases 3.2.1, 3.2.2, and 3.2.3 do not occur by Diagram~\ref{diagram:1027211}. 
\end{proof}

\begin{proposition}\label{prop:1217211}
Cases 3.3.1  and 3.3.2  do not occur.
\end{proposition}
\begin{proof}
By Diagram~\ref{diagram:1022212} or ~\ref{diagram:1022214}, we may assume that $a_{\alpha_0,\beta_0}\not\in K$.  The subsets
$K\cap K_{\alpha_0,\beta_0+k_1}$  and $K\cap K_{\alpha_0+\ell_1,\beta_0}$ are subsemiheaps of $K$ which fall into case 3.3.3, which is described below in Proposition~
~\ref{prop:0102221}. However, as shown in the proof of Lemma~\ref{lem:1230211} below, the four possible situations each lead to $a_{\alpha_0,\beta_0}\in K$.
\end{proof}

  \smallskip

\noindent{\bf Case 3.3.3} ($\overline{\beta}=\infty,  \overline{\alpha}=\infty,\overline{\gamma}=\infty$) 

\begin{quotation} 
Let $a_{\alpha_0,\beta_0+p}\in K$ with $p\ge 1$ and $a_{\alpha_0,\beta_0+p'}\not\in K$ for $1\le p'< p$.  Similarly,  let 
$a_{\alpha_0+r,\beta_0}\in K$ with $r\ge 1$ and $a_{\alpha_0+r',\beta_0}\not\in K$ for $1\le r'< r$ and  let 
$a_{\alpha_0+q,\beta_0+q}\in K$ with $q\ge 1$ and $a_{\alpha_0+q',\beta_0+q'}\not\in K$ for $1\le q'< q$

In the diagram below, the bullets represent the three points of $K$ which were just defined, the symbol $\circledcirc$ means that  the element $a_{\alpha_0\beta_0}$ may or may not belong to $K$, and the circles indicate that no element of $K$ occupies that position. The diagram represents just one of 13 possible cases (namely case (5) below), and is for illustration purposes only.

\bigskip

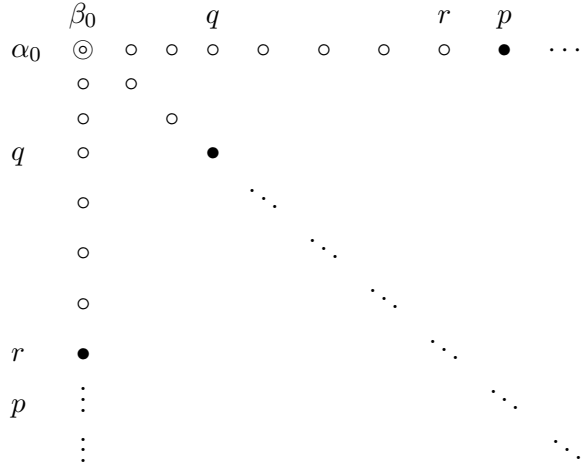
\begin{center}
\begin{tabular}{l c c c c c c c c c c c }
& $\beta_0$ & & & $q$ & &  & & $r$  & $p$  && \\
$\alpha_0$ &  $\circledcirc$ &  $\circ$&  $\circ$&  $\circ$&  $\circ$&  $\circ$&  $\circ$&  $\circ$
& $\bullet$ &$\cdots$  &  \\
& $\circ$ & $\circ$ & & & &   & &  &   && \\
&  $\circ$  && $\circ$ & &   & &   &   &   && \\
$q$ &  $\circ$  &&& $\bullet$ &    & &   &   &   && \\

&  $\circ$  &&&&$\ddots$&  &   &    &   && \\

 &  $\circ$  &&&&& $\ddots$ &      &   &   && \\

&  $\circ$  &&&&&& $\ddots$ &    &   && \\


$r$ &  $\bullet$  &&&&&&& $\ddots$ &   &    & \\

 $p$ &$\vdots$    &&&&&&&& $\ddots$       && \\

& $\vdots$  &&&&&&&&       & $\ddots$ & \\

\end{tabular}
\end{center}
\medskip

 Of course, we must consider the various relations between the three elements $p,q,r\in\IN$, some of which can be equal, of which there are six, namely
 \begin{itemize}
 \item $r\le p\le q$
 \item $p\le r\le q$
 \item $p\le q\le r$
 \item $r\le q\le p$
 \item $q\le r\le p$
 \item $q\le p\le r$
 \end{itemize}

  But for our purposes, it is necessary to distinguish 13 more refined cases, namely
 
 \begin{enumerate}
\item $r<p<q$
\item $p<r<q$

\item $p<q<r$

\item $r<q<p$

\item $q<r<p$

\item $q<p<r$

\item $r=p<q$
\item $p=q<r$
\item $r=q<p$
\item $r<p=q$
\item $p<r=q$
\item $q<r=p$
\item $r=p=q$
 \end{enumerate}
 
 \end{quotation}
 
 \begin{lemma}\label{lem:1112211}
Cases (3) to (11) do not occur.  If $a_{\alpha_0,\beta_0}\in K$, then cases (1) and (2) do not occur.
If $a_{\alpha_0\beta_0}\not\in K$, then case (12) does not occur. In case (13), $a_{\alpha_0\beta_0}\in K$
 \end{lemma}
 \begin{proof}By Lemma~\ref{lem:0925211}, we have 
 \begin{equation}\label{eq:1112211}
 a_{\alpha_0,\beta_0+p}a^*_{\alpha_0+q,\beta_0+q}a_{\alpha_0+r,\beta_0}=
  \left\{ \begin{array}{ll}
 \hbox{(i) } a_{\alpha_0,\beta_0+p-r} & \hbox{ if }r\le p\hbox{ and }q\le p\\
  \hbox{(ii) } a_{\alpha_0+r-p,\beta_0} & \hbox{ if }r\ge p\hbox{ and }q\le p\\
  \hbox{(iii) } a_{\alpha_0+r-p,\beta_0} & \hbox{ if }r\ge q\hbox{ and }q\ge p\\ 
 \hbox{(iv) } a_{\alpha_0+q-p,\beta_0+q-r} & \hbox{ if }r\le q\hbox{ and }q\ge p.
   \end{array}
   \right.
 \end{equation}
 
 \noindent In case (1) with $a_{\alpha_0\beta_0}\in K$, we obtain a contradiction by Diagram~\ref{diagram:1022214}. 
 
 \noindent In case (2) with $a_{\alpha_0\beta_0}\in K$, we obtain a contradiction by Diagram~\ref{diagram:1022212}.
  
\noindent In case (3), we obtain a contradiction by (\ref{eq:1112211}(iii)).
 
\noindent In case (4), we obtain a contradiction by (\ref{eq:1112211}(i)).
 
 \noindent In case (5), we obtain a contradiction by (\ref{eq:1112211}(i)).
 
\noindent  In case (6), we obtain a contradiction by (\ref{eq:1112211}(ii)).
 
\noindent  In case (7), we obtain a contradiction by (\ref{eq:1112211}(iv)).
 
 \noindent In case (8), we obtain a contradiction by (\ref{eq:1112211}(ii)).
 
\noindent In case (9), we obtain a contradiction by (\ref{eq:1112211}(i)).
 
 \noindent In case (10), we obtain a contradiction by (\ref{eq:1112211}(i)).
 
\noindent In case (11), we obtain a contradiction by (\ref{eq:1112211}(iii)).

\noindent In case (12) with  $a_{\alpha_0\beta_0}\not\in K$, 
we obtain a contradiction by (\ref{eq:1112211}(i)) or (ii).
 
 \noindent In case (13), $a_{\alpha_0\beta_0}\in K$ by (\ref{eq:1112211}(iii)).
 \end{proof}

It remains to consider cases (1) and (2), with $a_{\alpha_0\beta_0}\not \in K$, and the cases (12)  and (13), with  $a_{\alpha_0\beta_0}\in K$. The latter two are  resolved in Propositions~\ref{prop:1213211} and \ref{prop:1211211} and the former two in Lemma~\ref{lem:1230211}.

We start with some properties in case (12).  The basic diagram for case (12) is the following.
\medskip

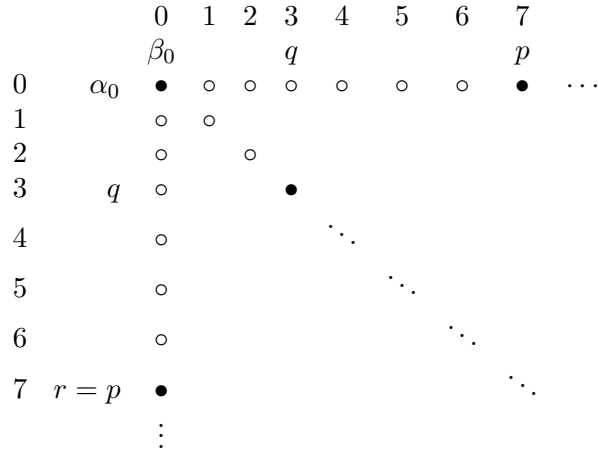
\begin{center}
\begin{tabular}{r r c c c c c c c c c c c }
&&0&1&2&3&4&5&6&7&\\
&& $\beta_0$  & & &$q$  &  & && $p$     && \\
0&$\alpha_0$ &  $\bullet$ &  $\circ$&  $\circ$&  $\circ$&  $\circ$&  $\circ$&  $\circ$
& $\bullet$  &$\cdots$ && \\
1&& $\circ$ & $\circ$ & & & &   & &  &   && \\
2&&  $\circ$  && $\circ$ & &   & &   &   &   && \\
3&$q$ &  $\circ$  &&& $\bullet$ &    & &   &   &   && \\

4&&  $\circ$  &&&& $\ddots$ & &   &    &   && \\

5& &  $\circ$  &&&&& $\ddots$ &      &   &   && \\

6&&  $\circ$  &&&&&& $\ddots$ &    &   && \\

7&$r=p$ &  $\bullet$ &&&&&&& $\ddots$ &   &    & \\

 &&  $\vdots$  &&&&&&&&     &  &   \\

\end{tabular}
\end{center}

\begin{lemma}\label{lem:1201211}
In case (12), with (necessarily) $a_{\alpha_0\beta_0}\in K$, 
\begin{description}
\item[(a)]
The rows $1, 2,\ldots, q-1$ contain no elements of $K$ above the diagonal

The columns $1, 2,\ldots, q-1$ contain no elements of $K$ below the diagonal
\smallskip

\item[(b)]
The points $a_{\alpha_0+q,\beta_0+q+i}$, for $1\le i\le p-q$ do not belong to $K$.

The points $a_{\alpha_0+q+i,\beta_0+q}$, for $1\le i\le p-q$ do not belong to $K$.
\smallskip

\item[(c)]
The points $a_{\alpha_0+q,\beta_0+p}$, and   $a_{\alpha_0+p,\beta_0+q}$
do not belong to $K$.
\smallskip

\item[(d)]

 The points $a_{\alpha_0+q,\beta_0+j}$, for $mp<j<mp+q$, with $m\in\IN$
do not belong to $K$.

The points $a_{\alpha_0+i,\beta_0+q}$, for $mp<i<mp+q$, with $m\in\IN$
do not belong to $K$.

\item[(e)] $a_{\alpha_0+q,\beta_0+p+q}, a_{\alpha_0+p+q,\beta_0+q}\in K$.
\smallskip

\item[(f)]
The points $a_{\alpha_0+q,\beta_0+j}$, for $mp+q<j\le (m+1)p$, with $m\in\IN$
do not belong to $K$.

The points $a_{\alpha_0+i,\beta_0+q}$, for $mp+q<i\le (m+1)p$, with $m\in\IN$
do not belong to $K$.
\smallskip

\item[(g)]
The points $a_{\alpha_0,\beta_0+j}$, for $mp<j< (m+1)p$, with $m\in\IN$
do not belong to $K$.

 The points $a_{\alpha_0+i,\beta_0}$, for $mp<i< (m+1)p$, with $m\in\IN$
do not belong to $K$.
\
\end{description}
\end{lemma}
\begin{proof}

In what follows, we shall  elaborate  on the above diagram.
 (In Diagram~\ref{diagram:0117221}, the symbols $\blacksquare,\blacktriangle,\blacktriangledown, \blacktriangleleft, \blacktriangleright, \blacklozenge, \spadesuit, \clubsuit$, and their blank versions, represent points which, {\it a priori}, do not belong to $K$. They should be temporarily ignored.   Also, Diagram~\ref{diagram:0117222} indicates the locations of (a)-(g).)

(a) Consider the two points $a_{\alpha_0\beta_0}$ and $a_{\alpha_0+i,\beta_0+j}$, the latter indicated by $\blacksquare$ in Diagram~\ref{diagram:0117221},  with $1\le i<q$, $2\le j<\infty$ and $i<j$, and suppose that $a_{\alpha_0+i,\beta_0+j}$ belongs to $K$.  Then by Diagram~\ref{diagram:1022211}, $x_3=a_{\alpha+i,\beta_0+i}$, indicated by $\square$, belongs to $K$, a contradiction. Therefore the rows $1,2,3,\ldots q-1$ contain no elements of $K$ above the diagonal.\smallskip

\begin{diagram}\label{diagram:0117221}
\begin{center}
\begin{tabular}{r rc c c c c c c c c c c c c c c c c}
&&0&1&2&3&4&5&6&7&8&9&10&11&12&13&14&15&16\\
&& $\beta_0$ & & && $q$  &  & &   $p$  &&&&$p+q$& &&2p&&\\
0&$\alpha_0$ &  $\bullet$ &  $\circ$&  $\circ$&  $\circ$&  $\triangleleft$&  $\triangledown$&  $\circ$
& $\bullet$ &  &  &$\clubsuit$ & &&&& & $\cdots$ \\
1&&$\circ$  &$\circ$  & & & &   & &  &   && &&   & &  & &$\cdots$ \\
2&&$\circ$ & & $\square$& & &   & &$\lozenge$  &   & &$\blacksquare$&&&&&&$\cdots$ \\
3&& $\circ$ & &&$\triangle$ &&&&&&&&&&&&&$\cdots$ \\
4&$q$ &  $\circ$ & &$\boxminus$&& $\bullet$   & $  $&  $\blacktriangledown$ &$\blacktriangleleft$   & $  $&  $\blacklozenge$ &$  $  &  $\bullet$ &  $  $ & $  $&$  $  &  $  $&$\cdots$\\
5&& $\circ$ &&&&&&&&&&&&&&&&$\cdots$ \\
6&&$\circ$&&&&$\spadesuit$&& $\ddots$&&&&&&&&&&$\cdots$ \\
7&$r=p$&  $\bullet$  &&&&$\blacktriangleright$ &    &  &$\bullet$  &    & &  & &    &  &   &   &$\cdots$ \\
8&&  &   &  &$\blacktriangle$ &  &    &  &   &    & &  & &    &  &   &   &$\cdots$ \\
9&&  &   &  &  & $\boxplus$ &    &  &   &    &$\ddots$&  & &    &  &   &   &$\cdots$ \\
10&& $\bigstar$&   &  &  &  &    &  &   &    & &  & &    &  &   &   &$\cdots$ \\
11&$p+q$ && & & &$\bullet$ &   &  &  &  &    &  &$\bullet$  &    & &  & &  $\cdots$ \\
12&&  &   &  &  &  &    &  &   &    & &  & &    &  &   &   &$\cdots$ \\
13&&  &   &  &  &  &    &  &   &    & &  & &    &  &   &   &$\cdots$ \\
 14&$2p$ &    & & &   &  &    &    & $\bullet$    &  &    &   &    &    & &$\bullet$  & &$\cdots$\\
 \end{tabular}
\end{center}
\end{diagram}
\medskip

Consider the two points $a_{\alpha_0\beta_0}$ and $a_{\alpha_0+i,\beta_0+j}$, the latter indicated by $\blacktriangle$ in Diagram~\ref{diagram:0117221}, with $1\le j<q$, $2\le i<\infty$ and $i>j$, and suppose that $a_{\alpha_0+i,\beta_0+j}$ belongs to $K$.  Then by Diagram~\ref{diagram:1022213}, $x_2=a_{\alpha+j,\beta_0+j}$, indicated by $\triangle$, belongs to $K$, a contradiction. Therefore  columns $1,2,3,\ldots q-1$ contain no elements of $K$ below the diagonal.\smallskip

(b) Assuming that $a_{\alpha_0+q,\beta_0+q+j}$, indicated by $\blacktriangledown$ in Diagram~\ref{diagram:0117221}, with $1\le j<p-q$, belongs to $K$, we have that 
$$K\supset K_{\alpha_0+q,\beta_0+q}^j
=\{a_{\alpha_0+q+\ell j,\beta_0+q+mj}: \ell,m\ge 0\}.
$$
Then by Lemma~\ref{lem:0925211}(i),
$$
a_{\alpha_0,\beta_0+p}a^*_{\alpha_0+q+\ell j,\beta_0+q+mj}a_{\alpha_0\beta_0}=a_{\alpha_0,\beta_0+(\ell-m)j+p}\in K,
$$
provided that $0\le(\ell-m)j+p$ and $q+mj\le p$.
Then with $\ell=0$ and $m=1$, we have that $a_{\alpha_0,\beta_0-j+p}$, indicated by $\triangledown$, belongs to $K$, which is a contradiction.  \smallskip

 For the second statement of (b), the proof is the same, namely, assuming that $a_{\alpha_0+q+i,\beta_0+q}$, indicated by $\spadesuit$, with $1\le i<p-q$, belongs to $K$, we have that 
$$K\supset K_{\alpha_0+q,\beta_0+q}^i
=\{a_{\alpha_0+q+\ell i,\beta_0+q+mi}: \ell,m\ge 0\}.
$$
Then by Lemma~\ref{lem:0925211}(i),
$$
a_{\alpha_0,\beta_0+p}a^*_{\alpha_0+q+\ell i,\beta_0+q+mi}a_{\alpha_0\beta_0}=a_{\alpha_0,\beta_0+(\ell-m)i+p}\in K,
$$
provided that $0\le(\ell-m)i+p$ and $q+mi\le p$.
Then with $\ell=0$ and $m=1$, we have that $a_{\alpha_0,\beta_0-i+p}$, indicated by $\triangledown$, belongs to $K$, which is a contradiction. \smallskip

(c) Assuming that $a_{\alpha_0+q,\beta_0+p}$, indicated by $\blacktriangleleft$,  belongs to $K$, then by Lemma~\ref{lem:0925211}(iii), 
$$
a_{\alpha_0,\beta_0+p}a^*_{\alpha_0+q,\beta_0+p}a_{\alpha_0+q,\beta_0+q}=
a_{\alpha_0,\beta_0+q},
$$
indicated by $\triangleleft$, belongs to $K$,
a contradiction.
Assuming that $a_{\alpha_0+p,\beta_0+q}$, indicated by $\blacktriangleright$,  belongs to $K$, then applying Diagram~\ref{diagram:1022214} to the two points $a_{\alpha_0+q,\beta_0+q}$ and $\blacktriangleright$ we have $x_1=a_{\alpha_0+q,\beta_0+p}\in K$, a contradiction to the previous paragraph.\smallskip

(d) Suppose $mp<j<mp+q$ and assume that $a_{\alpha_0+q,\beta_0+j}$, indicated by $\blacklozenge$ (with $m=1$),  belongs to $K$. Then by Lemma~\ref{lem:0925211}(iii),
$$
a_{\alpha_0,\beta_0+mp}a^*_{\alpha_0+q,\beta_0+j}a_{\alpha_0+q,\beta_0+q}=a_{\alpha_0+j-mp,\beta_0+q},
$$
indicated by $\lozenge$, belongs to $K$, a contradiction to (i), since $j-mp<q$.\smallskip

 Suppose $mp<i<mp+q$ and assume that $a_{\alpha_0+i,\beta_0+q}$, indicated by $\boxplus$ (with $m=1$),  belongs to $K$. Then by Lemma~\ref{lem:0925211}(iv),
$$
a_{\alpha_0+q,\beta_0+q}a^*_{\alpha_0+i,\beta_0+q}a_{\alpha_0+mp,\beta_0}=a_{\alpha_0+q,\beta_0+i-mp},
$$
indicated by $\boxminus$, belongs to $K$, a contradiction to (i$^\prime$), since $i-mp<q$.\smallskip

(e) By Diagram~\ref{diagram:1027211} or ~\ref{diagram:1027213}, applied to the vertices $a_{\alpha_0+q,\beta_0+q}$ and $a_{\alpha_0+p,\beta_0}$, $x_1=a_{\alpha_0+q,\beta_0+p+q}\in K$, and then by Diagram~\ref{diagram:1027212}, $a_{\alpha_0+p+q,\beta_0+q}\in K$.\smallskip

In the proofs of (f) and (g), we assume with no loss of generality, that $\alpha_0=\beta_0=0$.
\smallskip

(f)
Suppose that $(q,j)\in K$ with $mp+q<j\le (m+1)p$. By Lemma~\ref{lem:0925211}(iii), $(0,mp)(q,j)^*(q,q)=(j-mp,q)$.  This is a contradiction to (b) since $p\ge j-mp>q$.

Suppose that $(i,q)\in K$ with $mp+q<i\le (m+1)p$. By Lemma~\ref{lem:0925211}(iv), $(q,q)(i,q)^*(mp,0)=(q,i-mp)$.  This is a contradiction to (b) since $p\ge i-mp>q$. 
\smallskip

(g) Supposing that $a_{\alpha_0,\beta_0+j}$, with $p< j<\infty$ and $j\not\in\{2p,3p,\ldots\}$, denoted by $\clubsuit$ in Diagram~\ref{diagram:0117221}, belongs to $K$, we apply Diagram~\ref{diagram:1022212} to $a_{\alpha_0,\beta_0+mp}$ and $a_{\alpha_0,\beta_0+j}$, where $mp<j<(m+1)p$ to get $x_2=a_{\alpha_0+j-mp,\beta_0+j}\in K$, a contradiction since $j-mp<p$.  Hence no element of $K$ occupies any position in the row determined by $\alpha_0$ except for the points $a_{\alpha_0,\beta_0+mp}$ for $m\in\IN_0$.

Supposing that $a_{\alpha_0+i,\beta_0}$, with $p< i<\infty$ and $i\not\in\{2p,3p,\ldots\}$, denoted by $\bigstar$, belongs to $K$, we apply Diagram~\ref{diagram:1022214} to $a_{\alpha_0+mp,\beta_0}$ and $a_{\alpha_0+i,\beta_0}$, where $mp<i<(m+1)p$ to get a contradiction.  Hence no element of $K$ occupies any position in the column determined by $\beta_0$ except for the points $a_{\alpha_0+\ell p,\beta_0}$ for $\ell\in\IN_0$.\end{proof}

\smallskip

We now have Diagram~\ref{diagram:0117222} for case (12) with $a_{\alpha_0,\beta_0}\in K$, and 
it is clear that $K\cap K_{q,q}$ is also in subcase (12), 
so it follows that $K=\bigcup_{i=0}^\infty K_{q_i,q_i}^p$, where $a_{\alpha_0+q_i,\beta_0+q_i}$ are  the points of $K$ lying on the diagonal with
\[
q=q_0<q_1<q_2<\cdots<q_n<q_{n+1}<\cdots.
\]

\begin{proposition}\label{prop:1213211}
In case (12), with (necessarily) $a_{\alpha_0,\beta_0}\in K$,  let $a_{\alpha_0+q_i,\beta_0+q_i}$, $0\le i<\infty$,be  the points of $K$ lying on the diagonal, such that
\[
q=q_0<q_1<q_2<\cdots<q_n<p\quad \hbox{ and }\quad p<q_{n+1}<q_{n+2}<\cdots.
\]
Then
\[
K= \bigcup_{i=0}^n K_{q_i,q_i}^p.
\]
\end{proposition}

\begin{proof}
We know that $K=\bigcup_{i=0}^\infty K_{q_i,q_i}^p$. We need to show that $\bigcup_{i=n+1}^\infty K_{q_i,q_i}^p\subset \bigcup_{i=0}^n K_{q_i,q_i}^p$. For this it suffices to show that each $q_{n+j}$ with $j\ge 1$ is congruent to some element of $\{q_0,q_1,\ldots q_n\}$,
 modulo $p$.
 
Let $(q_k+\ell p,q_k+mp)\in K_{q_k,q_k}^p$ for some $k\ge n+1$ with fixed $\ell,m$,
and let  $(\ell' p,m'p)\in K_{0,0}^p$ with variable $\ell',m'$.
By Lemma~\ref{lem:0925211}(i), 
\[
(q_k+\ell p,q_k+mp)(\alpha_0,\beta_0)^*(\ell' p,m'p)=(q_k+\ell p,q_k+(m'+m-\ell')p)\in K
\]
as long as     $q_k+mp\ge 0$ and $\ell' p\le q_k+mp$,  We now choose $\ell'$ such that 
$q_k=(\ell'-m)p+d$, where $\ell'-m\ge 1$ and $0\le d<p$. To check that $\ell' p\le q_k+mp$, we note that   $(\ell'-m) p=q_k-d\le q_k$.  We now have 
\[
(q_k+\ell p,q_k+(m'+m-\ell')p)=(d+(\ell+\ell'-m)p,d).
\]
Thus $(d+tp,d)\in K=\bigcup_{i=0}^\infty K_{q_i,q_i}^p$ for some $t\ge 0$, so that $(d+tp,d)=(q_i+rp,q_i+sp)$ for some $i,r,s\ge 0$. Hence $d+tp=q_i+rp$ and $d=q_i+sp$, so by subtraction $tp=(r-s)p$ and $d+(r-s)p=q_i+rp$ so that $d=q_i+sp$. Since $d<p$, $s=0$ and $d=q_i$ with $i\le n$
\end{proof}

\begin{diagram}\label{diagram:0117222}
\begin{center}
\begin{tabular}{r rc c c c l c c c c c c l c c c c c}
&&0&1&2&3&4&5&6&7&8&9&10&11&12&13&14&15&16\\
&& $\beta_0$ & & && $q$  &  & &   $p$  &&&&$p+q$& &&2p&&\\
0&$\alpha_0$ &  $\bullet$ &  $\circ$&  $\circ$&  $\circ$&  $\circ$&  $\circ$&  $\circ$
& $\bullet$ &  $\circ$&  $(g)$& $\circ$&$\circ$&$(g)$ &$\circ$&$\bullet$& $\circ$&$\cdots$  \\
1&& $\circ$ & $\circ$ & $(a)$&$\circ$ &$\circ$ &  $\circ$ & $\circ$&$\circ$  &  $\circ$ &$\circ$&$\circ$ &$\circ$&  $\circ$ & $\circ$&$\circ$  &  $\circ$&$\cdots$ \\
2&& $\circ$ & $(a)$ &$\circ$ & $(a)$&$\circ$ &  $\circ$ & $\circ$&$\circ$  &  $\circ$ &$\circ$ &$\circ$&$\circ$&  $\circ$ & $\circ$&$\circ$  &  $\circ$&$\cdots$ \\
3&& $\circ$ & $\circ$ &$(a)$ &$\circ$& $(a)$&  $\circ$  & $\circ$&$\circ$  &  $\circ$ &$\circ$&$\circ$ &$\circ$&  $\circ$ & $\circ$&$\circ$  &  $\circ$&$\cdots$ \\
4&$q$ &  $\circ$  &$\circ$&$\circ$&  $(a)$& $\bullet$   & $(b)$&  $\circ$ &$(c)$   & $\circ$&  $(d)$ &$\circ$  &  $\bullet(e)$ &  $\circ$ & $(f)$&$\circ$  &  $\circ$&$\cdots$\\
5&& $\circ$ & $\circ$ & $\circ$&$\circ$ &$(b)$ &  &&&&&&&&&&&$\cdots$ \\
6&& $\circ$ & $\circ$ & $\circ$&$\circ$&$\circ$ & & $\ddots$&&&&&&&&&&$\cdots$ \\
7&$r=p$&  $\bullet$  &$\circ$&$\circ$&$\circ$&$(c)$ &&&$\bullet$  &&&&&&&$\bullet$&&$\cdots$ \\
8&& $\circ$ & $\circ$ & $\circ$&$\circ$ &$\circ$ &&&& &&&&&&&&$\cdots$ \\
9&& $(g)$ & $\circ$ & $\circ$&$\circ$ &$(d)$ &&&& &$\ddots$&&&&&&&$\cdots$ \\
10&& $\circ$ & $\circ$ & $\circ$&$\circ$ &$\circ$ &&&&&&&&&&&&$\cdots$ \\
11&$p+q$ &$\circ$&$\circ$&$\circ$&$\circ$&$\bullet(e)$ &&&&&&&$\bullet$  &&&&&  $\cdots$ \\
12&&$(g)$&$\circ$&$\circ$&$\circ$&$\circ$&&&&&&&&  $\ddots$ &&&&$\cdots$ \\
13&&$\circ$&$\circ$&$\circ$& $\circ$&$(f)$ &&&&&&&&& &&&$\cdots$ \\
 14&$2p$ &  $\bullet$  &$\circ$&$\circ$&  $\circ$& $\circ$&  && $\bullet$    &&&&&&&$\bullet$  &  &$\cdots$\\
  15& &  $\circ$ &$\circ$&$\circ$&  $\circ$& $\circ$&  &&    &&&&&&& &$\ddots$  &$\cdots$\\
 16& &  $\vdots$ &$\vdots$&$\vdots$&  $\vdots$& $\vdots$&  && $\vdots$   &&&&$\vdots$&&&$\vdots$ &  &$\cdots$\\

 \end{tabular}
\end{center}

\end{diagram}

\begin{proposition}\label{prop:1211211}
In case (13), $K=K_{\alpha_0,\beta_0}^p$

\end{proposition}
\begin{proof} The diagram for case (13) is the following. (Temporarily ignore the symbols $\blacksquare,\blacktriangle,\blacktriangledown,\blacktriangleleft$)

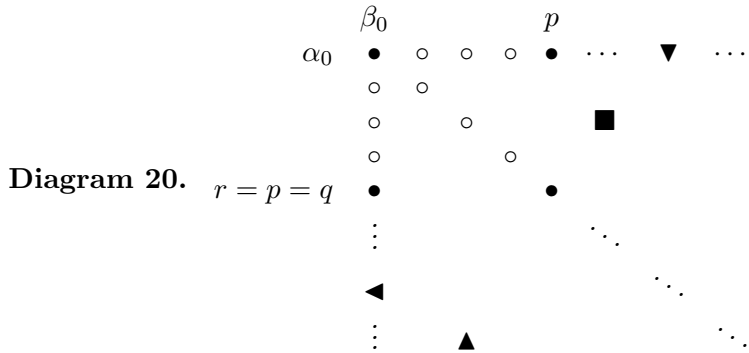
\begin{diagram}\label{diagram:0117223}
\begin{center}
\begin{tabular}{r c c c c c c c c c c c }
& $\beta_0$ & & &  &$p$ &  & &   &  \\
$\alpha_0$ & $\bullet$& $\circ$ &    $\circ$&  $\circ$&  $\bullet$&  $\cdots$&  $\blacktriangledown$&  $\cdots$&   \\
& $\circ$ & $\circ$ & & & &   & &  &    \\
&  $\circ$  && $\circ$ & &   &$\blacksquare$ &   &   &    \\
 &  $\circ$  &&& $\circ$ &    & &   &   &    \\

$r=p=q$&  $\bullet$  &&&& $\bullet$ & &   &    &    \\

 &  $\vdots$  &&&&& $\ddots$ &      &   &    \\

&  $\blacktriangleleft$  &&&&&& $\ddots$ &    &    \\

 &  $\vdots$  &&$\blacktriangle$&&&&& $\ddots$ &    \\




\end{tabular}
\end{center}
\end{diagram}
\smallskip

In the first place, we notice that by Diagrams~\ref{diagram:1022212} and ~\ref{diagram:1022214}, 
$$
K\supset K_{\alpha_0,\beta_0}^p=\{a_{\alpha_0+\ell p,\beta_0+mp}:\ell,m\in\IN_0\}.
$$
The next four paragraphs refer to Diagram~\ref{diagram:0117223}.

Supposing that $a_{\alpha_0+i,\beta_0+j}$ for $1\le i<p$ and $2\le j<\infty$, denoted by $\blacksquare$, belongs to $K$, we apply Diagram~\ref{diagram:1022211} to $a_{\alpha_0\beta_0}$ and 
$a_{\alpha_0+i,\beta_0+j}$  to get $x_3=a_{\alpha_0+i,\beta_0+i}\in K$, a contradiction.  Hence no element of $K$ occupies any position above the diagonal in the rows determined by $\alpha_0+i$, for  $1\le i<p$.
\smallskip

Supposing that $a_{\alpha_0+i,\beta_0+j}$ for $2\le i<\infty$ and $1\le j<p$, denoted by $\blacktriangle$, belongs to $K$, we apply Diagram~\ref{diagram:1022213} to $a_{\alpha_0\beta_0}$ and 
$a_{\alpha_0+i,\beta_0+j}$  to get $x_2=a_{\alpha_0+j,\beta_0+j}\in K$, a contradiction.  Hence no element of $K$ occupies any position below the diagonal in the columns determined by $\beta_0+j$, for  $1\le j<p$
\smallskip

Supposing that $a_{\alpha_0,\beta_0+j}$, with $p< j<\infty$ and $j\not\in\{2p,3p,\ldots\}$, denoted by $\blacktriangledown$, belongs to $K$, we apply Diagram~\ref{diagram:1022212} to $a_{\alpha_0,\beta_0+kp}$ and $a_{\alpha_0,\beta_0+j}$, where $kp<j<(k+1)p$ to get $x_2=a_{\alpha_0+j-kp,\beta_0+j}\in K$, a contradiction since $j-kp<p$.  Hence no element of $K$ occupies any position in the row determined by $\alpha_0$ except for the points $a_{\alpha_0,\beta_0+mp}$ for $m\in\IN_0$.
\smallskip

Supposing that $a_{\alpha_0+i,\beta_0}$, with $p< i<\infty$ and $i\not\in\{2p,3p,\ldots\}$, denoted by $\blacktriangleleft$, belongs to $K$, we apply Diagram~\ref{diagram:1022214} to $a_{\alpha_0+kp,\beta_0}$ and $a_{\alpha_0+i,\beta_0}$, where $kp<i<(k+1)p$ to get a contradiction.  Hence no element of $K$ occupies any position in the column determined by $\beta_0$ except for the points $a_{\alpha_0+\ell p,\beta_0}$ for $\ell\in\IN_0$.
\smallskip

We now have

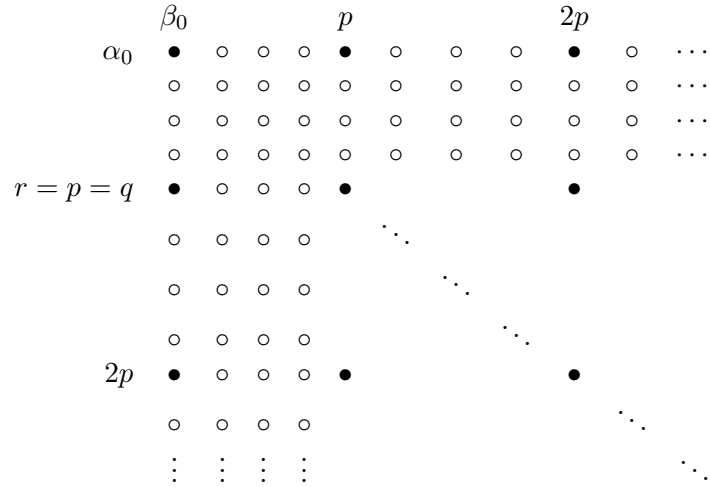
\begin{center}
\begin{tabular}{r c c c c c c c c c c c }
& $\beta_0$ & & &  &$p$ &  & & & $2p$   \\
$\alpha_0$ & $\bullet$& $\circ$ &    $\circ$&  $\circ$&  $\bullet$&  $\circ$&  $\circ$ &$\circ$&  $\bullet$ &$\circ$&$\cdots$ \\
& $\circ$ & $\circ$ &$\circ$ &$\circ$ &$\circ$ & $\circ$  &$\circ$ & $\circ$ &  $\circ$ & $\circ$ &  $\cdots$  \\
& $\circ$ & $\circ$ &$\circ$ &$\circ$ &$\circ$ & $\circ$  &$\circ$ & $\circ$ &  $\circ$ & $\circ$ &  $\cdots$  \\
& $\circ$ & $\circ$ &$\circ$ &$\circ$ &$\circ$ & $\circ$  &$\circ$ & $\circ$ &  $\circ$ & $\circ$ &  $\cdots$  \\

$r=p=q$&  $\bullet$  &$\circ$&$\circ$&$\circ$& $\bullet$ & &   &    &    $\bullet$ &&\\

 &  $\circ$  &$\circ$&$\circ$&$\circ$    && $\ddots$ &  &&    &   &    \\

 &  $\circ$  &$\circ$&$\circ$&$\circ$    &&& $\ddots$ &  &&       &    \\

 &  $\circ$  &$\circ$&$\circ$&$\circ$   && && $\ddots$ &      &   &    \\

$2p$ & $\bullet$ & $\circ$ & $\circ$ & $\circ$ & $\bullet$  &&&& $\bullet$ &     &\\

 &  $\circ$  &$\circ$&$\circ$&$\circ$    &&& &  &&   $\ddots$     &    \\

 &  $\vdots$  &$\vdots$&$\vdots$&$\vdots$    &&&  &  &&       & $\ddots$   \\

\end{tabular}
\end{center}
\smallskip

We next consider what happens in the row defined by $\alpha_0+p$.  

Supposing that $a_{\alpha_0+p,\beta_0+p+i}$ belongs to $K$, with $1\le i<p$,  then applying Diagram~\ref{diagram:1022213} to $a_{\alpha_0,\beta_0+p}$ and $a_{\alpha_0+p,\beta_0+p+i}$ we obtain $x_2=a_{\alpha_0+i,\beta_0+p+i}\in K$, 
which is a contradiction, and repeating this argument shows that no element of $K$ occupies any position in the row determined by $\alpha_0+p$ except for the points $a_{\alpha_0+ p,\beta_0+mp}$ for $m \in\IN_0$.

We next consider what happens in the column defined by $\beta_0+p$.  

Supposing that $a_{\alpha_0+p+i,\beta_0+p}$ belongs to $K$,  with $1\le i<p$, then applying Diagram~\ref{diagram:1022211} to $a_{\alpha_0+p,\beta_0}$ and $a_{\alpha_0+p+i,\beta_0+p}$ we obtain $x_3=a_{\alpha_0+p+i,\beta_0+i}\in K$, which is a contradiction, and repeating this argument shows that no element of $K$ occupies any position in the column determined by $\beta_0+p$ except for the points $a_{\alpha_0+ \ell p,\beta_0+p}$ for $\ell \in\IN_0$.

We now have (ignore temporarily the symbol $\blacksquare$)

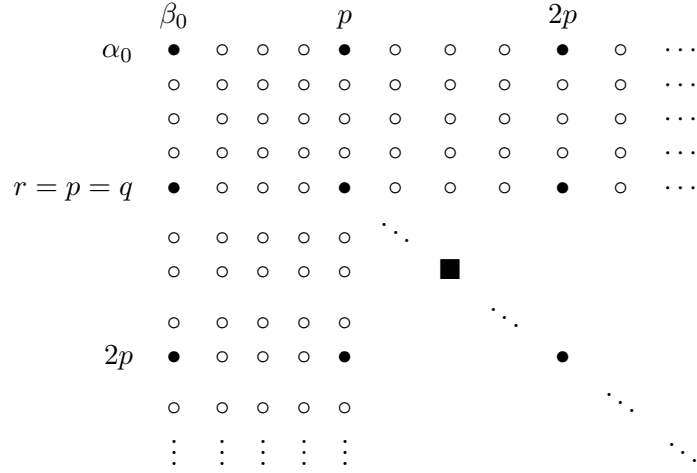
\begin{center}
\begin{tabular}{r c c c c c c c c c c c }
& $\beta_0$ & & &  &$p$ &  & & & $2p$   \\
$\alpha_0$ & $\bullet$& $\circ$ &    $\circ$&  $\circ$&  $\bullet$&  $\circ$&  $\circ$ &$\circ$&  $\bullet$ &$\circ$&$\cdots$ \\
& $\circ$ & $\circ$ &$\circ$ &$\circ$ &$\circ$ & $\circ$  &$\circ$ & $\circ$ &  $\circ$ & $\circ$ &  $\cdots$  \\
& $\circ$ & $\circ$ &$\circ$ &$\circ$ &$\circ$ & $\circ$  &$\circ$ & $\circ$ &  $\circ$ & $\circ$ &  $\cdots$  \\
& $\circ$ & $\circ$ &$\circ$ &$\circ$ &$\circ$ & $\circ$  &$\circ$ & $\circ$ &  $\circ$ & $\circ$ &  $\cdots$  \\

$r=p=q$&  $\bullet$  &$\circ$&$\circ$&$\circ$& $\bullet$ &$\circ$ &  $\circ$ & $\circ$   &    $\bullet$ &$\circ$&$\cdots$\\

 &  $\circ$  &$\circ$&$\circ$&$\circ$    &$\circ$ & $\ddots$ &  &&    &   &    \\

 &  $\circ$  &$\circ$&$\circ$&$\circ$    &$\circ$&& $\blacksquare$ &  &&       &    \\

 &  $\circ$  &$\circ$&$\circ$&$\circ$   &$\circ$& && $\ddots$ &      &   &    \\

$2p$ & $\bullet$ & $\circ$ & $\circ$ & $\circ$ & $\bullet$  &&&& $\bullet$ &     &\\

 &  $\circ$  &$\circ$&$\circ$&$\circ$    &$\circ$&& &  &&   $\ddots$     &    \\

 &  $\vdots$  &$\vdots$&$\vdots$&$\vdots$    &$\vdots$&&  &  &&       & $\ddots$   \\

\end{tabular}
\end{center}
\smallskip

Finally, we consider what happens along the diagonal.

Supposing that $a_{\alpha_0+p+i,\beta_0+p+i}$, denoted by $\blacksquare$, with $1\le i<p$, belongs to $K$, we apply Diagram~\ref{diagram:1022211} to $a_{\alpha_0+p,\beta_0}$ and $a_{\alpha_0+p+i,\beta_0+p+i}$,
we obtain $x_3=a_{\alpha_0+p+i,\beta_0+i}\in K$, which is a contradiction, and repeating this argument shows that no element of $K$ occupies any position in the diagonal  except for the points $a_{\alpha_0+ \ell p,\beta_0+\ell p}$ for $\ell \in\IN_0$.

We now have

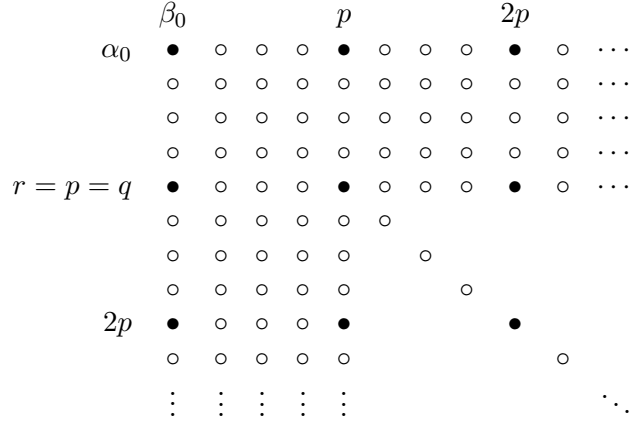
\begin{center}
\begin{tabular}{r c c c c c c c c c c c }
& $\beta_0$ & & &  &$p$ &  & & & $2p$   \\
$\alpha_0$ & $\bullet$& $\circ$ &    $\circ$&  $\circ$&  $\bullet$&  $\circ$&  $\circ$ &$\circ$&  $\bullet$ &$\circ$&$\cdots$ \\
& $\circ$ & $\circ$ &$\circ$ &$\circ$ &$\circ$ & $\circ$  &$\circ$ & $\circ$ &  $\circ$ & $\circ$ &  $\cdots$  \\
& $\circ$ & $\circ$ &$\circ$ &$\circ$ &$\circ$ & $\circ$  &$\circ$ & $\circ$ &  $\circ$ & $\circ$ &  $\cdots$  \\
& $\circ$ & $\circ$ &$\circ$ &$\circ$ &$\circ$ & $\circ$  &$\circ$ & $\circ$ &  $\circ$ & $\circ$ &  $\cdots$  \\

$r=p=q$&  $\bullet$  &$\circ$&$\circ$&$\circ$& $\bullet$ &$\circ$ &  $\circ$ & $\circ$   &    $\bullet$ &$\circ$&$\cdots$\\

 &  $\circ$  &$\circ$&$\circ$&$\circ$    &$\circ$ & $\circ$ &  &&    &   &    \\

 &  $\circ$  &$\circ$&$\circ$&$\circ$    &$\circ$&& $\circ$ &  &&       &    \\

 &  $\circ$  &$\circ$&$\circ$&$\circ$   &$\circ$& && $\circ$ &      &   &    \\

$2p$ & $\bullet$ & $\circ$ & $\circ$ & $\circ$ & $\bullet$  &&&& $\bullet$ &     &\\

 &  $\circ$  &$\circ$&$\circ$&$\circ$    &$\circ$&& &  &&   $\circ$     &    \\

 &  $\vdots$  &$\vdots$&$\vdots$&$\vdots$    &$\vdots$&&  &  &&       & $\ddots$   \\

\end{tabular} 
\end{center}
\smallskip

We are now in the position at the beginning of the proof, namely, the semiheap $K\cap K_{p,p}$ is in subcase (13) of case 3.3.3, and the result follows by applying successively what has already been proved.
\end{proof}

We shall now consider cases (1) and (2) with $a_{\alpha_0,\beta_0}\not\in K$ (See Lemma~\ref{lem:1112211}), and assume with no loss of generality, that $\alpha_0=\beta_0=0$. We consider Diagram~\ref{diagram:0127222} for case (1) and establish the following notation. The points of $K$ on the row determined by $\alpha_0$, indicated by $\blacktriangle$, are $a_{\alpha_0,\beta_0+m_i}$, with $1\le m_1<m_2< \cdots$, and the points  on the column determined by $\beta_0$, indicated by $\boxtimes$, are $a_{\alpha_0+\ell_i,\beta_0}$, with $1\le\ell_1<\ell_2<\cdots$.  We denote $\sigma=m_2-m_1$ and $\rho=\ell_2-\ell_1$. For example, in Diagram~\ref{diagram:0127222}, $\sigma=6$ and $\rho=r=4$.

\medskip
\begin{diagram}\label{diagram:0127222}

\begin{tiny}
\begin{center}
\begin{tabular}{| r | c |cc c c  c c c c c | c c c c c c c c c  c  cccc cc|}\hline
&& & & & & & & &  && $m_1$  & &&    &&  & $m_2$ & & &  &    & & &  &&   \\
&&0 &1&2 &  3& 4 & 5& 6& 7& 8& 9 &10 & 11& 12 & 13& 14& 15 & 16&17 & 18& 19& 20 & 21 & 22&23 &24 \\
&& & & & & $r$   & & & & & $p$ & & & $q$  & & & & & & & &  & &  &&\\\hline
0&&$\circ$ & $\circ$&$\circ$ &$\circ$ &$\circ$ & $\circ$&$\circ$ & $\circ$& $\circ$& $\blacktriangle$ &$\circ$ & $\circ$& $\circ$&$\circ$ &$\circ$ & $\blacktriangle$ & & & & & & $\blacktriangle$ & & &  \\
1&& $\circ$ &$\circ$ & & & & & & & & $\circ$& & & & & & & & & & & && & &  \\
2&&$\circ$ &&$\circ$  & & & & & & & $\circ$& & & & & & & && & & && &  &\\
3& &$\circ$ & & & $\circ$& & & & & & $\circ$& & & & & & & & & & && & &  &\\\hline
4&$\ell_1=r$ & $\boxtimes$ &$\circledcirc$ &$\circledcirc$ & $\circledcirc$& $\square$ & & & & $\square$ & $\circ$& & &  $\square$ & &&  & $\square$ & & & &$\square$ & & && $\square$\\
5&&$\circ$ & $\circleddash$& & & & $\circ$& & & &$\circ$ & & & & & & & & & & & && & & \\
6& & $\circ$& & $\circleddash$& & & & $\circ$& & & $\blacktriangle$ & & & & & & $\blacktriangle$ & & & & & & $\blacktriangle$ & & &  \\
7&& $\circ$& & & $\circleddash$  & &&&$\circ$&& & & & & & & & & & & & && &  &\\
8&$\ell_2$ & $\boxtimes$ & & & & $\square$ & & & & $\square$ &  && & $\square$ & &  & 
  &$\square$ && & &$\square$ && & & $\square$\\
9&$p$ & & & & & & & && & $\circ$& &  & & & & & & & & & &&& &  \\
10&& & & & & & & & &  & &$\circ$ & & && & & & & & & && & & \\
11&& & & & & & & & & & & & $\circ$& & & & & & & & & && & & \\
 12&$q$ & $\boxtimes$ & & & & $\square$ & & & & $\square$ &$\blacktriangle$ &  && $\square\bullet$ &  &  & $\blacktriangle$& $\square$& & && $\square$&$\blacktriangle$  &&& $\square$\\
 13& & & & & & & & &  &  & & & & &&&&& & & & && & & \\
14&& & & & & & & & & & & & & & & & & & & & & && &  &\\
15&& & & & & & &  &  & & & &  & & & &&&& && & & &  &\\
 16&& $\boxtimes$ & & & & $\square$ & & & & $\square$ & &&  & $\square$&
 &    &&$\square$ & & & &$\square$ &&& & $\square$\\
 17&& & & & & & & & & & & & & & & & & & & & & && &&  \\
  18&&  & & & &  & & & & &$\blacktriangle$ &  &  & &  &  & $\blacktriangle$& & & && &$\blacktriangle$  & &&\\
 19&& & & & & & & & & & & & & & & & & & & & & && & & \\
20& & $\boxtimes$ & & & & $\square$ & & & & $\square$ & & & & $\square$ & &&  & $\square$ & & & &$\square$  & &&& $\square$\\
21&& & & & & & &  & & & & &  & & & &&&&  && & & &&  \\
22& &  & & & &  & & & &  & &  &  & &  & & & & & && & && &\\
23& & & & & & & & & & & & & & & & & & & & & & && &  &\\
24& & $\boxtimes$ & & & & $\square$ & & & & $\square$ &$\blacktriangle$ &&  & $\square$  &  & & $\blacktriangle$&$\square$  &&& &$\square$  &$\blacktriangle$  &&& $\square$\\
25& &  & & & &  & & & &  & &  &  & &  & & & & & && & && &\\\hline

\end{tabular}
\end{center}
 \end{tiny}
 \end{diagram}
\medskip

We consider first $K_{r,0}$. By Diagram~\ref{diagram:1022212}, the points $(r,i)$, $1<i<\rho$, indicated by $\circledcirc$, do not belong to $K$. By Diagram~\ref{diagram:1022214}, $K\supset K_{r,0}^\rho$.
The semiheap $K\cap K_{r,0}$ falls into case 3.3.3, more precisely, either cases (7), (12) or (13), but case (7) does not occur.
In case (13),  the points $(r+j,j)$, $1\le j<\rho$, indicated by $\circleddash$,  do not belong to $K$, so by Proposition~\ref{prop:1211211}, $K\cap K_{r,0}=K_{r,0}^\rho$, and therefore in this case,
\[
K_1:=K_{r,0}^\rho\cap K_{r,p}=\{(\alpha_0+r+\ell\rho, \beta_0+m\rho: \ell\in\IN_0,m\rho\ge p    \}.
\].

By the same argument applied to  $K_{0,p}$, assuming that $K\cap K_{0,p}$ is also in case (13), we have 
\[
K_2:=K_{0,p}^\sigma\cap K_{r,p}=\{(\alpha_0+\ell'\sigma, \beta_0+p+m'\sigma):
\ell'\sigma\ge r, m'\in\IN_0    \}.
\]
$K_1$ is depicted by the  symbols $\square$ in $K_{r,p}$ and $K_2$ is depicted by  the symbols $\blacktriangle$ in $K_{r,p}$, and we must have $K_1=K\cap K_{r,p}=K_2$. 

As suggested by the diagram, we now show that  $\sigma=\rho$, and that $p$ and $r$  are divisible by $\sigma$. 
\begin{itemize}
\item Taking $m'=0$ and $\ell'\sigma\ge r$, $(\ell'\sigma,p)\in K_2$ so that $(\ell'\sigma,p)=(r+\ell\rho,m\rho)$ for some $\ell,m\in\IN_0$ with $m\rho\ge p$. Therefore $p=m\rho$ and $r=\ell'\sigma-\ell\rho$.
\item Taking $\ell=0$ and $m\rho\ge p$, $(r,m\rho)\in K_1$ so that $(r,m\rho)=(\ell'\sigma,p+m'\rho)$ for some $\ell',m'\in\IN_0$ with $\ell'\sigma\ge r$. Therefore $r=\ell'\sigma$ and $m\rho=p+m'\sigma$.

Thus $p$ is divisible by $\rho$, say $p=m_0\rho$, and $r$ is  divisible by $\sigma$, say $r=\ell_0\sigma$.

\item Taking $\ell=0$ and $(m_0+1)\rho=p+\rho>p$, 
$(r,(m_0+1)\rho)\in K_1$ so that 
$(r,(m_0+1)\rho)=(\ell''\sigma,p+m''\rho)$ for some $\ell'',m''\in\IN_0$ with $\ell''\sigma\ge r$. Therefore $r=\ell''\sigma$ and $(m_0+1)\rho=p+m''\sigma$. So $p+\rho=p+m''\sigma$, and $\rho=m''\sigma$.
\item Taking $m'=0$ and $(\ell_0+1)\sigma=r+\sigma> r$,
 $((\ell_0+1)\sigma,p)\in K_2$ so that $((\ell_0+1)\sigma,p)=(r+\ell\rho,m\rho)$ for some $\ell,m\in\IN_0$ with $m\rho\ge p$. Therefore $r+\sigma=r+\ell\rho$, so that $\sigma=\ell\rho$

Thus $\rho$ is divisible by $\sigma$ and $\sigma$ is divisible by $\rho$, hence $\sigma=\rho$.
\end{itemize}

Since $p$ and $r$ are each a multiple of $\rho$, it follows that  $(r,p)\in K$, so that $(0,p)(r,p)^*(r,0)=(0,0)\in K$, which is a contradiction.
We conclude that if both semiheaps $K\cap K_{r,0}$ and $K\cap K_{0,p}$ are in case (13), then case (1) does not occur.  

It remains to show that case (1) does not occur in the three other possible cases, namely,
\begin{itemize}
\item $K\cap K_{r,0}$ is in case (12) and $K\cap K_{0,p}$ is in case (13)
\item $K\cap K_{r,0}$ is in case (13) and $K\cap K_{0,p}$ is in case (12)
\item $K\cap K_{r,0}$ is in case (12) and $K\cap K_{0,p}$ is in case (12)
\end{itemize}

Let us now suppose that $K\cap K_{r,0}$ is in case (12),  and $K\cap K_{0,p}$ is in case (13) and refer to Diagram~\ref{diagram:0127223}.  Recall that the points of $K$ on the row determined by $\alpha_0$, indicated by $\blacktriangle$, are $a_{\alpha_0,\beta_0+m_i}$, with $1\le m_1<m_2< \cdots$, and the points  on the column determined by $\beta_0$, indicated by $\boxtimes$, are $a_{\alpha_0+\ell_i,\beta_0}$, with $1\le\ell_1<\ell_2<\cdots$.  We denote $\sigma=m_2-m_1$ and $\rho=\ell_2-\ell_1$. For example, in Diagram~\ref{diagram:0127223}, $\sigma=6$ and $\rho=r=4$.

Since $K\cap K_{r,0}$ is assumed in case (12), by Proposition~\ref{prop:1213211},
there exist $0=j_0<1\le j_1<j_2<\cdots<j_n<\rho$ such that
 \[
K\cap K_{r,0}=\bigcup_{i=0}^n K_{r+j_i,j_i}^\rho,
\]
and therefore in this case,
\[
K_1:=K_{r,0}^\rho\cap K_{r,p}=\bigcup_{i=0}^n\left(K_{r+j_i,j_i}^\rho\cap K_{r,p}\right)
=\bigcup_{i=0}^n\{(r+j_i+\ell\rho,j_i+m\rho):\ell,m\in\IN_0, j_i+m\rho\ge p \}.
\]
In Diagram~\ref{diagram:0127223}, we indicate the points of $K_{r+j_1,j_1}^\rho=K_{r+3,3}^4$ with the symbols $\heartsuit$, and the points of 
$K_{r+j_2,j_2}^\rho=K_{r+6,6}^4$ with the symbols $\oplus$.

As before, assumng that $K\cap K_{0,p}$ is in case (13), we have 
\[
K_2:=K_{0,p}^\sigma\cap K_{r,p}=\{(\alpha_0+\ell'\sigma, \beta_0+p+m'\sigma):
\ell'\sigma\ge r, m'\in\IN_0    \}
\]

In Diagram~\ref{diagram:0127223} below,
$K_1$ is depicted by the  symbols $\square,\heartsuit,\oplus$ in $K_{r,p}$ and $K_2$ is depicted by  the symbols $\blacktriangle$ in $K_{r,p}$, and we must have $K_1=K\cap K_{r,p}=K_2$. 

\begin{diagram}\label{diagram:0127223}
\begin{tiny}
\begin{center}
\begin{tabular}{| r | c |cc c c  c c c c c | c c c c c c c c c  c  ccccc cc|}\hline
&& & &&$j_1$  & & & $j_2$&  && $m_1$  & &&    &&  & $m_2$ & & &  &    & & &  &&  & \\
&&0 &1&2 &  3& 4 & 5& 6& 7& 8& 9 &10 & 11& 12 & 13& 14& 15 & 16&17 & 18& 19& 20 & 21 & 22&23 & 24& 25\\
&& & & & & $r$   & & & & & $p$ & & & $q$  & & & & & & & & & & & & & \\\hline
0&&$\circ$ & $\circ$&$\circ$ &$\circ$ &$\circ$ & $\circ$&$\circ$ & $\circ$& $\circ$& $\blacktriangle$ &$\circ$ & $\circ$& $\circ$&$\circ$ &$\circ$ & $\blacktriangle$ & & & & & & $\blacktriangle$ & & & & \\
1&& $\circ$ &$\circ$ & & & & & & & & $\circ$& & & & & & & & & & & && & & &\\
2&&$\circ$ &&$\circ$  & & & & & & & $\circ$& & & & & & & && & & && & &&\\
3& &$\circ$ & & & $\circ$& & & & & & $\circ$& & & & & & & & & & && & & & &\\\hline
4&$\ell_1=r$ & $\boxtimes$ &&&& $\square$ & & & & $\square$ & $\circ$&  &   & $\square$&  & & & $\square$ & & & &$\square$ && &&$\square$&\\
5&&$\circ$ & & & & & $\circ$& & & &$\circ$ & & & & & & & & & & & && & && \\
6& & $\circ$& & & & & & $\circ$& & & $\blacktriangle$ & & & & & & $\blacktriangle$ & & & & & & $\blacktriangle$ & & & & \\
7&$r+j_1$& $\circ$& & &  $\heartsuit$ & &&&$\heartsuit\circ$&&& &$\heartsuit$ & & & & $\heartsuit$& & & &$\heartsuit$ & & &&$\heartsuit$ & & \\
8&$\ell_2$ & $\boxtimes$ & & & & $\square$ & & & & $\square$ & &  & & $\square$   & 
  & && $\square$& &&& $\square$ & & &&$\square$&\\
  9&$p$ & & & & & & & && & $\circ$& &  & & & & & & & & & &&& &  &\\
10&$r+j_2$& & & & & & &$\oplus$ & &  & &$\oplus\circ$ & & && $\oplus$& & & &$\oplus$ & & && $\oplus$& & &\\
11&& && &$\heartsuit$ & & & & $\heartsuit$& &  & & $\heartsuit\circ$& && & $\heartsuit$& & & & $\heartsuit$& & &&$\heartsuit$ & & \\
 12&$q$ & $\boxtimes$ & & & & $\square$ & & & & $\square$ &$\blacktriangle$ &  & &$\square\bullet$ &  &  & $\blacktriangle$& $\square$& &&&$\square$  &$\blacktriangle$  & &&$\square$&\\
 13& & & & & & & & &  &  & & & & &&&&& & & & && & & &\\
14&& & & & & & & $\oplus$& & & & $\oplus$& & & & $\oplus$& & & &$\oplus$ & & &&$\oplus$ &  &&\\
15&& & && $\heartsuit$& & & &$\heartsuit$  &  & & & $\heartsuit$&  & & & $\heartsuit$&&&&$\heartsuit$ && & & $\heartsuit$&  &\\
 16&& $\boxtimes$ & & & & $\square$ & & & & $\square$ & &  & 
 &  $\square$  && & &$\square$ & & &&$\square$& & &&$\square$&\\
17&& & & & & & & & & & & & & & & & & & & & & && && & \\
  18&&  & & & &  & &$\oplus$ & & &$\blacktriangle$ & $\oplus$ &  & &  & $\oplus$ & $\blacktriangle$& & & $\oplus$&& &$\blacktriangle$  & $\oplus$&&&\\
 19&& && &$\heartsuit$ & & & & $\heartsuit$& & & &$\heartsuit$ & & & & $\heartsuit$& & & & $\heartsuit$& & && $\heartsuit$& & \\
20&& $\boxtimes$ & & & & $\square$ & & & & $\square$ & &  & 
 &  $\square$  && & &$\square$ & & &&$\square$& & &&$\square$&\\
21&& & & & & & &  & & & & &  & & & &&&&  && & & && & \\
22& &  & & & &  & & $\oplus$& &  & & $\oplus$ &  & &  &$\oplus$ & & & &$\oplus$ && & &$\oplus$& &&\\
23& && & & $\heartsuit$& & & & $\heartsuit$& & & &$\heartsuit$ & & & & $\heartsuit$& & & &$\heartsuit$ & & && $\heartsuit$&  &\\
24& & $\boxtimes$ & & & & $\square$ & & & & $\square$ &$\blacktriangle$ &  &  & $\square$&  & & $\blacktriangle$&$\square$  & & && $\square$&$\blacktriangle$  & &&$\square$&\\
25&& & &  & & && & && & &   && & &  & & & & &&& & & &\\\hline
\end{tabular}
\end{center}
 \end{tiny}
 \end{diagram}
\medskip

We show first that $\rho=\sigma$.
Since $K_1\subset K_2$, for  $\ell,m,i\in\IN_0$, with $j_i+m\rho\ge p$, there exist $\ell',m'\in\IN_0$ with $\ell'\sigma\ge r$ and 
\begin{equation}\label{eq:1231213}
(r+j_i+\ell\rho,j_i+m\rho)=(\ell'\sigma,p+m'\sigma).
\end{equation}

Fix $i$ such that $j_i\ge p$.  Then  
 for  all $\ell,m\in\IN_0$,  there exist $\ell',m'\in\IN_0$ with $\ell'\sigma\ge r$ such that 
\[
r+j_i+\ell\rho=\ell'\rho\hbox{ and }j_i+m\rho=p+m'\sigma.
\]
Eliminating $j_i$ from these two equations results in 
\begin{equation}\label{eq:1231211}
r+p=(m-\ell)\rho+(\ell'-m')\sigma,
 \end{equation}
 with $(\ell',m')$ depending on $(\ell,m)$ and satisfying $\ell'\sigma\ge r$.
 
Since $K_2\subset K_1$, for  $\ell,m\in\IN_0$, with $\ell\sigma\ge r$, there exist $\ell',m',i\in\IN_0$ with $j_i+m'\rho\ge p$ such that 
 \[
r+j_i+\ell'\rho=\ell\sigma\hbox{ and }j_i+m'\rho=p+m\sigma.
\]
Eliminating $j_i$ from these two equations results in 
\begin{equation}\label{eq:1231212}
r+p=(m'-\ell')\rho+(\ell-m)\sigma,
 \end{equation}
 with $(\ell',m')$ depending on $(\ell,m)$, provided $\ell\sigma\ge r$.
 
  With $\ell\ge 0$ and $m\ge 0$, from (\ref{eq:1231211}), there exist $\ell_1,m_1$ such that 
 \[
 r+p=(m-\ell)\rho+(\ell_1-m_1)\sigma
 \]
and there exist $\ell_2,m_2$ such that 
 \[
 r+p=(m+1-\ell)\rho+(\ell_2-M_2)\sigma,
 \]
 so by subtraction, $0=\rho+[(\ell_2-m_2)+(\ell_1-m_1)]\sigma$ and $\sigma$ divides $\rho$.

 With $\ell\sigma\ge r$ and $m\ge 0$, from (\ref{eq:1231212}), there exist $\ell_3,m_3$ such that 
 \[
 r+p=(m_3-\ell_3)\rho+(\ell-m)\sigma
 \]
and there exist $\ell_4,m_4$ such that 
 \[
 r+p=(m_4-\ell_4)\rho+(\ell+1-m)\sigma,
 \]
 so by subtraction, $0=[(m_4-\ell_4)-(m_3-\ell_3)]\rho+\sigma$ and $\rho$ divides $\sigma$.

Hence $\rho=\sigma$ and from (\ref{eq:1231211}) or (\ref{eq:1231212}), $\rho$ divides $r+p$. Now, taking $i=0,\ell=0$ in (\ref{eq:1231213}), $r=\ell'\sigma$, so that also $\rho$ divides $p$. Hence, as in the previous case, $(r,p)\in K$, so that $(0,p)(r,p)^*(r,0)=(0,0)\in K$, which is a contradiction.
We conclude that if the semiheap $K\cap K_{r,0}$ is in case (12) and the semiheap $K\cap K_{0,p}$ is in case (13), then case (1) does not occur.  

 Since the adjoint mapping is an anti-isomorphism of the extended bicyclic semigroup (See Remark~\ref{rem:0315221}), it follows that if the semiheap $K\cap K_{r,0}$ is in case (13) and the semiheap $K\cap K_{0,p}$ is in case (12), then case (1) does not occur. 
 
 It remains to consider the case when both semiheaps $K\cap K_{r,0}$ and 
 $K\cap K_{0,p}$ are in case (12). After this, again since the adjoint mapping is an anti-isomorphism, and case (1) has been shown to not occur , it will follow that case (2) also does not occur, so we will have the following lemma.
 
 \begin{lemma}\label{lem:1230211}
Cases (1) and (2) with (necessarily) $(0,0)\not\in K$, do not occur.
\end{lemma}
 \begin{proof}
 It suffices to show that if both semiheaps $K\cap K_{r,0}$ and 
 $K\cap K_{0,p}$ are in case (12), then $(0,0)\in K$ and therefore case (1) does not occur. We have 
 \[
 K_1:=K\cap K_{r,0}\cap K_{r,p}=\bigcup_{i=0}^n\{(r+j_i+\ell\rho,j_i+m\rho): \ell,m\in\IN_0, j_i+m\rho\ge p\}
 \]
and
\[
 K_2:=K\cap K_{0,p}\cap K_{r,p}=\bigcup_{i=0}^{n'}\{(k_i+\ell\sigma,p+k_i+m\sigma):\ell,m\in\IN_0, k_i+\ell\sigma\ge r\},
 \]
where $0=k_0<1\le k_1<k_2<\cdots<k_{n'}<\sigma$.

Since $K_1\subset K_2$, for $i,\ell,m\in\IN_0$ with $j_i+m\rho\ge p$, there exist $i',\ell',m'\in\IN_0$ satisfying $k_{i'}+\ell'\sigma\ge r$, such that 
\[
(r+j_i+\ell\rho,j_i+m\rho)=(k_{i'}+\ell'\sigma,p+k_{i'}+m'\sigma)
\]
so that
\[
r+j_i+\ell\rho=k_{i'}+\ell'\sigma\hbox{ and }      j_i+m\rho=p+k_{i'}+m'\sigma
\]
Fix $i$ such that $j_i\ge p$. Then  for every $\ell,m\in\IN_0$,  by subtraction, we have
\begin{equation}\label{eq:0101211}
r+p=(m-\ell)\rho+(\ell'-m')\sigma
\end{equation}
with $\ell',m'$ depending only on $\ell,m\in\IN_0$ (and $\ell'$ satisfying $k_{i'}+\ell'\sigma\ge r$ for some $i'$). \smallskip

Since $K_2\subset K_1$, for $i,\ell,m\in\IN_0$ with $k_i+\ell\sigma\ge r$, there exist $i',\ell',m'\in\IN_0$ satisfying $j_{i'}+m'\rho\ge p$, such that 
\[
(k_i+\ell\sigma, p+k_i+m\sigma)=(r+j_{i'}+\ell'\rho,j_{i'}+m'\rho)
\]
so that
\[
k_i+\ell\sigma=r+j_{i'}+\ell'\rho\hbox{ and }    p+k_i+m\sigma=j_{i'}+m'\rho\]

Fix $i$ such that $k_i\ge r$. Then  for every $\ell,m\in\IN_0$,  by subtraction, we have
\begin{equation}\label{eq:0101212}
r+p=(\ell-m)\sigma+(m'-\ell')\rho
\end{equation}
with $\ell',m'$ depending only on $\ell,m\in\IN_0$ (and $m'$ satisfying $j_{i'}+m'\rho\ge p$ for some $i'$). \smallskip

  With $\ell\ge 0$ and $m\ge 0$, from (\ref{eq:0101211}), there exist $\ell_1,m_1$ such that 
 \[
 r+p=(m-\ell)\rho+(\ell_1-m_1)\sigma
 \]
and there exist $\ell_2,m_2$ such that 
 \[
 r+p=(m+1-\ell)\rho+(\ell_2-M_2)\sigma,
 \]
 so by subtraction, $0=\rho+[(\ell_2-m_2)+(\ell_1-m_1)]\sigma$ and $\sigma$ divides $\rho$.\smallskip

  With $\ell\ge 0$ and $m\ge 0$, from (\ref{eq:0101212}), there exist $\ell_1,m_1$ such that 
 \[
 r+p=(\ell-m)\sigma+(m_1-\ell_1)\rho
 \]
and there exist $\ell_2,m_2$ such that 
 \[
 r+p=(\ell+1-m)\sigma+(m_2-\ell_2)\rho,
 \]
 so by subtraction, $0=\sigma+[(m_2-\ell_2)+(m_1-\ell_1)]\rho$ and $\rho$ divides $\sigma$.

Hence $\rho=\sigma$ and from (\ref{eq:0101211}) or (\ref{eq:0101212}),
$\sigma$ divides $p+r$. In fact, $\sigma$ divides both $p$ and $r$. Indeed, since $(q,q)\in K_1$ and $K_1=K_2$, there exist $\ell,m,i$ and $\ell',m',i'$, such that 
\[
(q,q)=(r+j_i+\ell\sigma,j_i+\sigma) = (k_{i'}+\ell'\sigma,p+k_{i'}+m'\sigma),
\]
so that $r+\ell\sigma=m\sigma$ and $p+m'\sigma=\ell'\sigma$. We now have $(r,p)\in K$, so that $(0,0)=(0,p)(r,p)^*(r,0)\in K$, a contradiction.
\end{proof}

We summarize the results of Lemma~\ref{lem:1112211} to Lemma~\ref{lem:1230211} in the following proposition.
\begin{proposition}\label{prop:0102221}
If the semiheap $K$ is in case 3.3.3, then either  $K=K_{\alpha_0,\beta_0}^p$ for some $p>0$, or
there exists $p>0$ such that
\[
K= \bigcup_{i=0}^n K_{q_i,q_i}^p.
\]
where $a_{\alpha_0+q_i,\beta_0+q_i}$, $0\le i<\infty$, are  the points of $K$ lying on the diagonal, such that
\[
q=q_0<q_1<q_2<\cdots<q_n<p\quad \hbox{ and }\quad p<q_{n+1}<q_{n+2}<\cdots.
\]
\end{proposition}
\begin{proof}
By Lemma~\ref{lem:1230211}, cases (1) and (2) do not occur.  
By Lemma~\ref{lem:1112211}, cases (3)-(11) do not occur. Cases (12) and (13) are described in Propositions~\ref{prop:1213211} and \ref{prop:1211211}.
\end{proof}

The following theorem is the main result of this paper, listing all of the subsemiheaps of the extended bicyclic semigroup.
\begin{theorem}\label{thm:0211221}
The subsemiheaps of the extended bicyclic semigroup are  the inductive limits (see Remark~\ref{rem:0308221}) of sequences of  the following semiheaps $K$:
\begin{itemize}
\item $K$ is a single point $\{a_{pq}\}$ (Lemmal~\ref{lem:1225211})
\item $K=
\{a_{\alpha_0,\beta_0}\}\cup\{a_{\alpha_0+k_j,\beta_0+k_j}:1\le j\le n_0\}$,
 where  $1\le k_1<k_2<\cdots<k_{n_0}$ and  $a_{\alpha_0,\beta_0}$ and $a_{\alpha_0+k_j,\beta_0+k_j}$  are the only elements of $K$ that are in the diagonal. (Proposition~\ref{prop:1225211})
 \item $K=\{a_{\alpha_0\beta_0}\}\cup\{a_{\alpha_0+k_j,\beta_0+k_j}:j\in\IN\}$ (Proposition~\ref{prop:1225212})
 \item
 There exist $\sigma,\ell_0\in\IN$, such that 
\[
  K=\{a_{\alpha_0\beta_0}\}\cup\{a_{\alpha_0+k_j,\beta_0+k_j}:j=1,\ldots \ell_0-1\}\cup K_{\alpha_0+k_{\ell_0},\beta_0+k_{\ell_0}}^\sigma
  \]
  or
\[  K=\{a_{\alpha_0\beta_0}\}\cup\{a_{\alpha_0+k_i,\beta_0+k_i}:1\le i<\ell_0\}\cup 
 \left(\bigcup_{i=\ell_0}^{j} K^{\sigma}_{\alpha_0+k_i,\beta_0+k_i}\right)
    \]
 where $1\le k_1<k_2<\cdots<k_n<\cdots <\infty$, $a_{\alpha_0,\beta_0}$ and $a_{\alpha_0+k_i,\beta_0+k_i}$  are the only elements of $K$ that are in the diagonal. (Proposition~\ref{prop:1225212})
\item
 $K=K_{\alpha_0,\beta_0}^p$ for some $p>0$ (Proposition~\ref{prop:0102221})
 \item There exists $p>0$ such that
\[
K= \bigcup_{i=0}^n K_{q_i,q_i}^p.
\]
where $a_{\alpha_0+q_i,\beta_0+q_i}$, $0\le i<\infty$, are  the points of $K$ lying on the diagonal, such that
\[
q=q_0<q_1<q_2<\cdots<q_n<p\quad \hbox{ and }\quad p<q_{n+1}<q_{n+2}<\cdots.
\]
(Proposition~\ref{prop:0102221})
\end{itemize}
\end{theorem}

\section{Injectivity of W*-TROs}

The notation for this section is the following.\smallskip

$\IN=1,2,\ldots$;
$\IN_0=\IN\cup\{0\}$;
$\IZ=\IN_0\cup -\IN$

$S$ is an inverse semigroup with generalized inverse $x^*$

$K$ is a subset of $S$ closed under the triple product $xy^*z$ (semiheap).  \smallskip

$\pi$ is the left regular representation of $S$ on $H:=\ell^2(S)$ so that $S$ is an orthonormal basis for $H$ and $\pi(x)$ is the partial isometry defined by 
$
\pi(x)y=xy 
$
if $yy^*\le x^*x$ and $\pi(x)y=0$ otherwise.\smallskip

$C^*_{\hbox{red}}(S)$ is the C*-algebra generated by $\{\pi(x):x\in S\}$ and is the norm closure of $\hbox{span}\, \pi(S)$. \smallskip

 $\hbox{TRO}(K)$ is the TRO generated by $\pi(K)$ and is the norm closure of $\hbox{span}\, \pi(K)$.  \smallskip
 
 $\hbox{VN}(S)$ is the von Neumann algebra generated by $\pi(S)$  and is the weak closure of $C^*_{\hbox{red}}(S)$ \smallskip
 
 $\hbox{VNTRO}(K)$ is the W*-TRO generated by $\pi(K)$ and is the weak closure of   $\hbox{TRO}(K)$.\smallskip

Details of the left regular representation are as follows (\cite[pp. 25--27]{Paterson1999}).  We have
$$
\pi(a_{ij})a_{pq}=\left\{ \begin{array}{cc} a_{ij}a_{pq}, & a_{pq}a_{pq}^*\le  a_{ij}^*a_{ij}\\
0,& \hbox{otherwise}
\end{array}
\right.,
$$
that is,
$$
\pi(a_{ij})a_{pq}=\left\{ \begin{array}{cc} a_{ij}a_{pq}, & a_{pp}\le  a_{jj}\\
0,& \hbox{otherwise}
\end{array}
\right.,
$$
or
$$
\pi(a_{ij})a_{pq}=\left\{ \begin{array}{cc} a_{ij}a_{pq}, & p\ge j\\
0,& \hbox{otherwise}
\end{array}
\right.,
$$
or
$$
\pi(a_{ij})a_{pq}=\left\{ \begin{array}{cc} a_{i+p-j,q}, & p\ge j\\
0,& \hbox{otherwise}
\end{array}
\right.
$$

Define provisionally a linear map $\Phi_0:\hbox{span}\, \pi(S)\rightarrow \hbox{span}\, \pi(K)$ as follows: $\Phi_0(0)=0$, and 
for $x_1,\ldots x_n\in S$,
\[
\Phi_0\left(\sum_{i=1}^n\lambda_i\pi(x_i)\right)=
\sum_{x_i\in K}\lambda_i\pi(x_i).
\]

\begin{proposition}\label{prop:0908211}
The idempotent map $\Phi_0$ is well-defined and contractive, and therefore extends to a contractive projection $\Phi$ on $C^*_{\hbox{red}}(S)$ with range $\hbox{TRO}(K)$.  Moreover, $\Phi$ extends to a completely contractive projection on $\hbox{VN}(S)$ with range $\hbox{VNTRO}(K)$.  Hence, if $\hbox{VN}(S)$ in an injective von Neumann algebra, then $\hbox{VNTRO}(K)$ is an injective operator space.
\end{proposition}
\begin{proof}
Let $a=\left\| \sum_{i=1}^n\lambda_i\pi(x_i)   \right\|$ and $b=\left\| \sum_{x_i\in K}\lambda_i\pi(x_i)  \right\|$. With $\xi=\sum_{z\in S}(\xi,z)z\in\ell^2(S)$, $$\pi(x_i)\xi=\sum_{zz^*\le x_i^*x_i}(\xi,z)x_iz$$ so that
\[
b^2=\sup_{\|\xi\|\le1}\left\| \sum_{x_i\in K}\lambda_i\pi(x_i) \xi \right\|^2=\sup_{\|\xi\|\le1}\sum_{x_i\in K,z\in S,zz^*\le x_i^*x_i}|\lambda_i(\xi,z)|^2
\]
and by the same calculation
\[
a^2=\sup_{\|\xi\|\le1}\left\| \sum_{x_i\in S}\lambda_i\pi(x_i) \xi \right\|^2=\sup_{\|\xi\|\le1}\sum_{x_i\in S,z\in S,zz^*\le x_i^*x_i}|\lambda_i(\xi,z)|^2.
\]
Therefore $\Phi_0$ is contractive and extends to a contractive projection on $C^*_{\hbox{red}}(S)$ with range $\hbox{TRO}(K)$.

Let $A=C^*_{\hbox{red}}(S)$, $U=TRO(K)$, so that $\Phi^{**}$ is a contractive projection on the von Neumann algebra $A^{**}$ with range $U^{**}$. By \cite[Lemma]{LanRus83}, $U^{**}$ is isomorphic to $\hbox{VNTRO}(K)$, and by \cite[Theorem 2.5]{EffOzaRua01}, $\Phi^{**}$ is a completely contractive projection with range $\hbox{VNTRO}(K)$. Therefore, $\overline{\Phi}:=\Phi^{**}|_{VN(S)}$ is a completely contractive projection of $VN(S)$ onto $VNTRO(K)$.  If $VN(S)$ is injective, then there is a completely contractive projection $P$ of $B(H)$ onto $VN(S)$, so that $\overline{\Phi}\circ P$ is a completely contractive projection with range $VNTRO(K)$.
\end{proof}

\begin{example}\label{exam:0916211}
Suppose that $e$ and $f$ are idempotents in the inverse semigroup $S$ and that $K=eSf$, which is a subsemiheap of $S$. The corresponding induced map takes the form $\pi(x)\mapsto \pi(exf)$ (with $0\rightarrow 0$) and is contractive since
\[
\left\|\sum_i \lambda_i \pi(ex_if)\right\|\le \|\pi(e)\|\left\|\sum_i\lambda_i\pi(x_i)\right\|\|\pi(f)\|=\left\|\sum_i\lambda_i\pi(x_i)\right\|.
\]
Hence Proposition~\ref{prop:0908211} applies.  This also applies to maps of the form $x\mapsto ex$ and $x\mapsto xf$.
\end{example}

The maximal subgroups of any inverse semigroup $S$ are of the form 
\[
S_e^e=\{s\in S:ss^*=s^*s=e\}
\]
for some idempotent $e$ (See \cite[p. 198]{Paterson1999}). Thus, the maximal subgroups of  $E$ in Example~\ref{exam:0916212} reduce to one-element groups, so are trivially amenable and hence by \cite[Theorem 4.5.2]{Paterson1999}, $\hbox{VN}(E)$ is injective.\footnote{As pointed out to the authors by Alan Paterson,  the proof of \cite[Theorem 4.5.2]{Paterson1999} required the assumption that the universal groupoid of the inverse semigroup be Hausdorff. This assumption holds for the extended bicyclic semigroup $E$, because it is an E-unitary semigroup (see \cite[Corollary 3.7]{SkandalisEtAlJRAM02} and \cite[p.57]{Lawson2017}). In addition, it appears from \cite{SkandalisEtAlJRAM02} that the Hausdorff assumption can actually be dropped.}

Since $VN(E)$ is injective, where $E$ is the extended bicyclic semigroup,  it follows from Proposition~\ref{prop:0908211} and  Example~\ref{exam:0916211}, that  $VNTRO(a_{ii}Ea_{jj})$ is an injective operator space, as are $VNTRO(Ea_{jj})$ and $VNTRO(a_{ii}E)$. 
More generally, we have
\begin{corollary}\label{cor:0211221}
All of the subsemiheaps of the extended bicyclic semigroup $E$
 (which were determined in Theorem~\ref{thm:0211221}) give rise to injective W*-TROs.
 \end{corollary}

\end{document}